\documentclass[11pt,leqno,oneside,a4paper]{amsart}
\usepackage{amsmath}
\usepackage{amsthm}
\usepackage{amsfonts}
\usepackage{amssymb}
\usepackage{xcolor}
\usepackage{enumitem}
\usepackage[mathcal]{euscript}
\setlist[enumerate]{label={\rm(\roman*)}}

\usepackage{mathtools}
\usepackage{hyperref}
\usepackage{natbib}
\hypersetup{
	breaklinks=true,
	colorlinks=true,
	linkcolor=teal,
	citecolor=violet}
\usepackage{graphicx}

\usepackage[
	textsize=tiny,
	textwidth=3.2cm,
	colorinlistoftodos,
	backgroundcolor=teal!30!white,
	linecolor=teal!50!white
	]{todonotes}

\colorlet{darkgreen}{green!80!black}

\theoremstyle{plain}
\newtheorem{theorem}{Theorem}[section]
\newtheorem{lemma}[theorem]{Lemma}
\newtheorem{corollary}[theorem]{Corollary}

\newtheorem*{proposition*}{Proposition}

\theoremstyle{definition}
\newtheorem{remark}[theorem]{Remark}

\newtheorem{definition}[theorem]{Definition}

\newtheorem{convention}[theorem]{Convention}

\numberwithin{theorem}{section}
\numberwithin{equation}{section}

\DeclarePairedDelimiter\abs{\lvert}{\rvert}

\renewcommand\d{\mathrm{d}}
\newcommand\dd{\,\d}

\def\emphdef{\emph}

\newcommand{\N}{\mathbb{N}}
\newcommand{\R}{\mathbb{R}}

\newcommand{\RR}{\mathcal R}

\newcommand{\MM}{\mathcal M}

\newcommand{\lgn}{\mathcal L}
\def\M{\mathcal M}

\newcommand{\ip}{{\frac{1}{p}}}
\newcommand{\ipo}{{\frac{1}{p_0}}}
\newcommand{\ipi}{{\frac{1}{p_i}}}
\newcommand{\ipon}
    {{\frac{1}{p_1}}}

\newcommand{\RM}{(\RR,\mu)}

\newcommand{\MRM}{\MM\RM}

\newcommand{\rn}{\R^n}

\def\ls{\lesssim}

\newcommand{\Wlgn}{W_{\kern-1.2pt\lgn}\kern0.7pt}

\newcommand{\pe}\relax
\newcommand{\pb}\relax
\newcommand{\emphm}\emph

\usepackage{xspace}
\makeatletter
\DeclareRobustCommand\onedot{\futurelet\@let@token\@onedot}
\def\@onedot{\ifx\@let@token.\else.\null\fi\xspace}

\makeatother

\makeatletter
\def\paragraph{\bigskip\@startsection{paragraph}{4}%
  \z@\z@{-\fontdimen2\font}%
  {\normalfont\bfseries}}
\makeatother

\title{Interpolation of classical Lorentz spaces measuring oscillation}
\author{Amiran Gogatishvili, Julio S. Neves, Lubo\v s Pick and Hana Tur\v cinov\'a} 

\address{Amiran Gogatishvili, Institute of Mathematics of the Czech Academy of Sciences, \v Zitn\'a~25, 115~67 Praha~1, Czech Republic, ORCiD 0000-0003-3459-0355\\
E-mail: gogatish@math.cas.cz}

\address{J\'ulio S. Neves,
University of Coimbra, CMUC, Department of Mathematics, Apartado
3008, EC Santa Cruz, 3001-454 Coimbra, Portugal, ORCiD 0000-0002-7675-6862
\\
E-mail jsn@mat.uc.pt}

\address{Lubo\v s Pick, Department of Mathematical Analysis, Faculty of Mathematics and Physics, Charles University, Sokolovsk\'a~83, 186~75 Praha~8, Czech Republic, ORCID 0000-0002-3584-1454\\
E-mail: pick@karlin.mff.cuni.cz}

\address{Hana Tur\v cinov\'a, Czech Technical University in Prague, Faculty of Electrical Engineering, Department of Mathematics, Technick\'a~2, 166~27 Praha~6, Czech Republic; ORCID 0000-0002-5424-9413\\
E-mail: hana.turcinova@fel.cvut.cz}

\date{\today}

\thanks{This research was supported in part by  the Fundação para a Ciência e a Tecnologia (Portuguese Foundation for Science and Technology) under the scope of the projects UID/00324/2025 (https://doi.org/10.54499/UID/00324/2025) (Centre for Mathematics of the University of Coimbra)  and by the grant 23-07420S of the Czech Science Foundation. The research of A. Gogatishvili
was partially supported by the  Institute of Mathematics, CAS  RVO:67985840 and by Shota
Rustaveli National Science Foundation (SRNSF), grant no: FR22-17770. Hana Turčinová was supported in part by the Danube Region Grant no.~8X25005 (SOFTA) of the Czech Ministry of Education, Youth and Sports.}

\begin{document}

\begin{abstract}
    We obtain an explicit characterization of the $K$-functional of a pair of weighted classical Lorentz spaces of type $S$. We develop a method for obtaining such characterization based on a relation between the desired quantity and the $K$-functional of a specific couple of spaces of type $\Lambda$, which are substantially more manageable than their companions of type $S$. The core of our techniques is a subtle manipulation with respective fundamental functions. We present several applications, in particular we nail down a formula for the $K$-functional of a Lebesgue space and a classical Lorentz space of type $S$ with a power weight, and using this formula we establish an inequality of a reverse Marchaud type.
\end{abstract}

\maketitle

{\parindent=0pt\textbf{The paper is dedicated to Professor Hans Triebel, the true classic of the field of function spaces, at the occasion of his 90th birthday, with compliments from the authors.}}

\section{Introduction and main results}

Many important inequalities in analysis have the form
\begin{equation}
    \label{E:basic-estimate}
    (Af)^{**}(t)\le C(Af)^{*}(t) + \text{something},
\end{equation}
in which $A$ is some operator, $f\mapsto f^*$ denotes the operation of taking the non-increasing rearrangement of a real-valued function $f$ acting on some nonatomic $\sigma$-finite measure space, and, for $t\in(0,\infty)$, $f^{**}(t)$ is the integral mean of $f^*$ over the interval $(0,t)$. In the particular case when $C=1$, one can subtract $(Af)^{*}(t) $ from both sides of~\eqref{E:basic-estimate} and obtain thereby a useful pointwise upper estimate for the quantity $(Af)^{**}-(Af)^{*}$. 

A pioneering result in this direction (with $A$ being the identity operator) was obtained by Bennett, DeVore and Sharpley in~\cite{Ben:81}, where the space `Weak $L^{\infty}$' was introduced as the collection of all functions $f$ such that $f^{**}-f^{*}$
is bounded. It was shown that Weak $L^{\infty}$ can effectively replace $L^{\infty}$ when $L^{\infty}$ does not give satisfactory results, and, in particular, its interesting interpolation properties were found. It was also observed that Weak $L^{\infty}$ is not a linear set. It was identified as the rearrangement hull of the space $\operatorname{BMO}$ of functions having bounded mean oscillation.

In~\cite[Theorem~2]{Bag:86}, Bagby and Kurtz established a weighted version of an estimate of the form~\eqref{E:basic-estimate} and used it to gain control over a maximal singular integral defined by a~Calder\'on--Zygmund kernel in terms of the Hardy--Littlewood maximal operator. This technique eventually lead to the establishment of the so-called \emph{better $\lambda$-inequality}, a significant enhancement of the classical \emph{good $\lambda$-inequality} of Burkholder and Gundy (\cite[Theorem~2]{Bur:72}), which was first used to get estimates of the distribution function of the area integral in terms of the nontangential maximal function, and which was later exploited many times in a variety of situations, for example by Coifman in~\cite{Coi:72} or Coifman and Fefferman in  \cite{Coi:74}. It was however demonstrated in~\cite{Bag:86} that the `better $\lambda$' inequality contains more information and provides sharper results.

In~\cite[Equation (A7), page~219]{Alv:89}, Alvino, Trombetti and Lions established the inequality 
\begin{equation}
    \label{E:kolyada-inequality}
    f^{**}(t)-f^*(t)
    \le C_n t^{\frac{1}{n}}
    |\nabla f|^{**}(t)
    \quad\text{for $t\in(0,\infty)$,}
\end{equation}
which is valid for every locally integrable function $f\colon\rn\to\R$ whose first-order weak derivatives are locally integrable and
where $C_n$ is a dimensional constant.
The authors pointed out a connection to certain optimization problems involving gradients with prescribed rearrangements. In particular, they showed that this inequality can be used in order to give a short proof of a sharp version of the Sobolev--Lorentz inequality. An inequality equivalent to~\eqref{E:kolyada-inequality}, in which the functional $f^{**}(t)-f^*(t)$ is replaced by $f^{*}(t)-f^*(2t)$, was independently discovered by Kolyada in~\cite[Lemma 5.1]{Kol:89}. In \cite[Inequality~(3.1)]{Kol:07}), the inequality is established in the form
\begin{equation}
    \label{E:kolyada-inequality-2}
    f^{**}(t)-f^*(t)
    \le \sqrt{n} t^{\frac{1}{n}}
    |\nabla f|^{**}(t)
    \quad\text{for $t\in(0,\infty)$.}
\end{equation}
Inequalities in this spirit were intensively investigated by Mart\'{\i}n and Milman in~\cite{Mil:07}, where among other results it was shown that the inequality
\begin{equation}
    \label{E:martin-milman}
    f^{**}(t)-f^*(t)
    \le C_n t^{1-\alpha}
    |\nabla f|^{**}(t)
    \quad\text{for $t\in(0,\infty)$}
\end{equation}
for $\alpha\in(0,1)$ characterizes that the underlying domain belongs, owing to its isoperimetric profile, to the Maz'ya class $\mathcal J_{\alpha}$. Iterated versions are applied to higher-order derivatives and to obtain Sobolev--Poincaré inequalities including fractional versions. 

In~\cite[Inequality (3.24)]{Ker:06}, an estimate of the form~\eqref{E:basic-estimate} is given for a specific supremum-type operator, which plays a crucial role in the so-called \emph{reduction principle} and provides characterization of optimal partner spaces in sharp Sobolev embeddings.  
The crucial role of the quantity $f^{**}-f^{*}$ in limiting Sobolev embeddings was however spotted earlier, namely in~\cite{Bas:03}, and subsequently corroborated in~\cite{Mil:04} and~\cite{Pus:05}. A yet different quantity, namely $f^*(\frac t2)-f^*(t)$, intimately related to $f^{**}(t)-f^*(t)$, has been used to define various function classes of interest in specific applications. In~\cite{Mal:02}, this functional was utilized in order to describe a new scale of function spaces, one of which significantly enhanced the target function space in a limiting Sobolev embedding. It turned out that the first-order Sobolev space with the degree of integrability coinciding with the dimension of the underlying domain continuously embeds into this new space, providing a sharper result than embeddings that had been known before such as those from~\cite{Tru:67,Poh:65,Yud:61,One:63,Pee:66,Hun:66,Han:79,Bru:79,Bre:80,Maz:11,Cwi:98}. However, both the quantities are very close to one another, and sometimes they are equivalently interchangeable - see~\cite[Theorem~4.1]{Bas:03} or~\cite[Inequalities~(2.8) and~(2.12)]{Kol:07}. 

The quantity $f^{**}-f^{*}$ turned out to be of vital importance in the search of an explicit characterization of the optimal target classes for Sobolev inequalities (\cite[Theorem~A]{Ker:09}, \cite[Theorem~1.2]{Kub:26}), and it played similarly important role in trace inequalities~\cite{Cia:08} or potential estimates~\cite{Edm:20}.

The above examples clearly show that the expression $f^{**}-f^*$ has potential to define interesting function spaces. On the top of it, it turns out to be of particular use in the theory of the \emph{classical Lorentz spaces}. Let us recall that these spaces are of three types, which we will respectively refer to as of type $\Lambda$, of type $\Gamma$, and  of type $S$. For precise definitions and the summary of their basic functional analytic properties, see Section~\ref{S:preliminaries} below.

The discovery of the connection to classical Lorentz spaces is due to Sinnamon~\cite[Theorem~4.1]{Sin:02}, who used the quantity $f^{**}-f^*$ to give a new formula for the associate norm  of a space of type~$\Gamma$. Other (equivalent) characterizations were obtained for instance in~\cite{Gog:03,Gog:14}, but the formula in~\cite{Sin:02} for the associate norm is unique in the sense that it breaks the associate norm
into two parts corresponding to the size, represented by $f^*$, and the smoothness, represented by $f^{**}-f^{*}$, of a function $f$.

Roughly speaking, the spaces of types $\Lambda$, $\Gamma$ and $S$ are governed by the degree of (weighted) integrability of quantities $f^*$, $f^{**}$, and $ f^{**}-f^{*}$, respectively, where $f$ is a given function.
Of the threesome, the spaces of the type $S$ are the youngest, the most difficult to handle, and also the most interesting. They were introduced only in 2005 (\cite{Car:05}), in sharp contrast to those of the type  $\Gamma$, which first appeared in 1990 (\cite{Saw:90}) and in an even sharper contrast to those of the type  $\Lambda$, which were introduced already in 1951 (\cite{Lor:51}). Their study is so difficult mainly owing to the fact that these spaces systematically avoid nice functional properties. In particular, they are, in general, neither quasinormable, nor necessarily complete, as sets they are not necessarily linear, and they do not possess the lattice property (see~\cite{Car:05} for more details).

The introduction of spaces of type $S$ was, of course, originally motivated by the ubiquitous appearance of the quantity $f^{**}-f^{*}$, mentioned above. They have been used in order to characterize various dual structures of function classes (\cite{Sin:02,Gog:03,Gog:12,Kol:07,Gog:14,Ker:09}), they served as an appropriate tool for characterizing validity of the bilinear weighted Hardy inequality for monotone functions (\cite{Kre:17}), and they played a vital role in the research of continuity properties of solutions to the $p$-Laplace system of elliptic equations (\cite{Alb:17}).  We shall now describe their intimate connection to the theory of interpolation.

In the theory of real interpolation and its numerous applications, a central issue is the question of characterization of the $K$-functional for a prescribed pair of function spaces. In order to give such a characterization, a lot of effort has been spent by many authors, and the topic occupies substantial parts in classical monographs such as~\cite{Ste:71,Ber:76,Kre:82,Ben:88,Bru:91,Tri:95}. It is also a subject of several more recent papers motivated by various important operators (such as fractional integrals, for instance) whose endpoint behavior is nonstandard and hence the classical theory fails (see e.g.~\cite{Gog:09,Mal:12,Bae:22,Kub:25}).

The indispensable role of the functional $f^{**}-f^{*}$ in interpolation theory is thoroughly described in~\cite[Chapter~5, Section~7]{Ben:88}, where it is observed, among many other results, that this functional can be used to show that the non-increasing rearrangement of a function oscillates no more than does the function itself.

The power of the interpolation properties of spaces of type $S$ was found soon after they were first introduced. In~\cite{Gog:12}, the inequality~\eqref{E:kolyada-inequality-2} was used to obtain sharp estimates of rearrangements of functions in terms of their moduli of continuity. This result was in turn applied to get
sharp embeddings of general Besov spaces into Lorentz spaces and a characterization of 
the rearrangement-invariant hull of a general Besov space. The approach was based on evaluating the $K$-functional for pairs of spaces $(X,S_X)$, where $X$ is a generic rearrangement-invariant space containing functions defined on a Lipschitz domain in an Euclidean space, and $S_X$ is a generalized classical Lorentz space of type $S$ built upon $X$.  For further applications in a similar but slightly different direction, see~\cite{Gog:24}. 

The techniques in~\cite{Gog:12} however have some limitations that quite adequately correspond to the state of the art of knowledge at the time of their origin, and they naturally propagate in turn into related results and applications. For instance, the structures appearing in theorems of~\cite{Gog:12} are assumed to be r.i.~spaces, in particular normed, which constitutes an unfortunate restriction ruling out many interesting function classes endowed merely with quasinorms. Furthermore, the Boyd indices of the present function spaces are required to be kept away from the endpoints, which, once again, bans plenty of interesting limiting cases.

In this paper we look at the problem from a different angle. We  take an approach based on some recent advances, leading to a substantial extension of the theory in which, in particular, the above-mentioned restrictions are chipped away. Our key theoretical achievement is a pointwise characterization of the $K$-functional for a pair of spaces of type $S$. 

Our argument proceeds in several separate steps. First, we reduce the problem to the question of characterizing the $K$-functional for the pair of spaces of type $\Lambda$ appropriately tied to the original pair. Next, we axiomatize a pair of abstract quasinormed spaces in a way that enables us to establish a formula for their $K$-functional. As the next step, we 
apply the abstract result to a pair of spaces of type $\Lambda$. The combination of the partial results then yields a solution of the principal problem. Finally, we point out some applications of the main result, including an inequality of reverse Marchaud type related to an embedding of a homogeneous-type Sobolev space into a Lebesgue space.

Before stating the results, we  need to fix some notation. We denote by $L^{1}_{\operatorname{loc}}(0,\infty)$ the collection of all functions integrable over every subset of $(0,\infty)$ of finite measure (in particular, integrable near zero). Following~\cite{Car:05}, we introduce the set 
\begin{equation*}
    A=\left\{f\in L^{1}_{\operatorname{loc}}(0,\infty)\ \text{nonnegative, non-increasing and satisfying $\lim_{t\to\infty}f(t)=0$}\right\}
\end{equation*}
and the operator $T$, defined by
\begin{equation}
    \label{E:T}
    Tf(t) = 
    \int_{0}^{\frac{1}{t}}
    f(s)\dd s
    - \frac{1}{t}f\left(\frac{1}{t}\right)
    \quad\text{for $f\in L^{1}_{\operatorname{loc}}(0,\infty)$ and $t\in(0,\infty)$.}
\end{equation}
It is easy to see that if $f\in L^{1}_{\operatorname{loc}}(0,\infty)$ is nonnegative and non-increasing on $(0,\infty)$, then  
\begin{equation*}
    \lim_{t\to\infty}Tf(t)=0
\end{equation*}
and $Tf$ itself is non-increasing on $(0,\infty)$. Indeed, this follows from
\begin{equation}
    \label{E:T-star-alternative}
    Tf(t) = 
    \int_{0}^{\frac{1}{t}}
    \left(f(s)
    -f(\tfrac{1}{t})\right)\dd s
    \quad\text{for $t\in(0,\infty)$.}
\end{equation} 
Moreover, it was observed in~\cite[Lemma 2.1]{Car:05} that $T$ is idempotent in the sense that
\begin{equation}
    \label{E:T-idempotent}
    T\circ T=\operatorname{Id}\quad\text{on $A$},
\end{equation}
and that $T$ is self-adjoint with respect to the $L^1(0,\infty)$-pairing in the sense that
\begin{equation*}
    \int_{0}^{\infty}Tf(t)g(t)\,dt
    = 
    \int_{0}^{\infty}Tg(t)f(t)\,dt
    \quad\text{for every $f,g\in A$.}
\end{equation*}
In particular, it follows immediately from~\eqref{E:T-idempotent} that every $f\in A$ can be represented as $f=Th$ for some $h\in A$.
It might be useful to add that we shall be mostly applying the operator $T$ to $f^*$, where $f$ is defined on some $\sigma$-finite measure space and $f^*(\infty)=0$. In that case, one has
\begin{equation}
    \label{E:T-star}
    Tf^*(t) = 
    \int_{0}^{\frac{1}{t}}
    f^*(s)\dd s
    - \tfrac{1}{t}f^*(\tfrac{1}{t})
    =\frac{1}{t}\left(f^{**}(\tfrac{1}{t})-f^*(\tfrac{1}{t})\right)
    \quad\text{for $t\in(0,\infty)$.}
\end{equation}

Similarly as in~\cite{Ari:90} and~\cite{Car:05}, given $p\in(0,\infty)$ and a \emph{weight} (that is, nonnegative and measurable function) $w$ on $(0,\infty)$, we say that $w$ satisfies the condition $\operatorname{B}_p$, and write $w\in \operatorname{B}_p$, if there exists a positive constant $C$ such that
\begin{equation*}
    t^p\int_{t}^{\infty}\frac{w(s)}{s^p}\dd s
    \le C
    \int_{0}^{t}w(s)\dd s
    \quad\text{for $t\in(0,\infty)$,}
\end{equation*}
and that $w$ satisfies the \emph{reverse} condition $\operatorname{B}_p$, denoted by $w\in\operatorname{RB}_p$, if there exists a positive constant $C$ such that
\begin{equation*}
    \int_{0}^{t}w(s)\dd s
    \le C
    t^p\int_{t}^{\infty}\frac{w(s)}{s^p}\dd s\quad\text{for $t\in(0,\infty)$.}
\end{equation*}
When $p\in(0,\infty)$ and a weight $w$ on $(0,\infty)$ are fixed, we denote by $\widetilde w$ the function defined by
\begin{equation}
    \label{E:tilde-w}
    \widetilde{w}(t)=
    t^{p-2}w(\tfrac1t)
    \quad\text{for $t\in(0,\infty)$.}
\end{equation}
One has, by a change of variables (cf.~\cite[page~382]{Car:05}), 
\begin{equation*}
    w\in \operatorname{RB}_p\quad
    \Leftrightarrow\quad
    \widetilde w\in\operatorname B_p.
\end{equation*} 

\begin{convention}
    \label{CON:representation}
    When a quasi-Banach space $X=X\RM$ allows representation in the sense that there is a quasinorm $\|\cdot\|_{\overline X}$ such that $\|f\|_X=\|f^*\|_{\overline{X}}$ for every admissible $f$, we will sometimes, when there is no danger of confusion, write $X$ instead of $\overline X$ for the representation space. It will be always clear from the context whether the function whose norm is taken is defined on a general measure space or on an interval. Typically, all rearrangements, characteristic functions of intervals, and also all arguments and images of the operator $T$, are always considered to be defined on $(0,\infty)$. We will use the barred notation only in connection with abstract spaces.
\end{convention}

Our first result relates the $K$-functional of a pair of spaces of type $S$ to that of a pair of spaces of type $\Lambda$ with the help of the operator $T$ from~\eqref{E:T}. 

\begin{theorem}
    \label{T:1.1}
    Assume that $p_0,p_1\in(0,\infty)$, and let $w_0,w_1$ be weights on $(0,\infty)$. Define the functions $\psi_0,\psi_1\colon (0,\infty)\to[0,\infty]$ by
    \begin{equation}
        \label{E:varphi}
        \psi_i(t)
        =
        \left(\int_{t}^{\infty}s^{-p_i}w_i(s)\dd s
        \right)^{\frac{1}{p_i}}
        \quad\text{for $t\in(0,\infty)$ and $i\in\{0,1\}$.}
    \end{equation}
    Assume that there exists a positive constant $C$ such that
    \begin{equation}
        \label{E:cond1}
        \psi_i(t)
        \le C
        \psi_i(2t)
        \quad\text{for $t\in(0,\infty)$ and $i\in\{0,1\}$.}
    \end{equation}
    Suppose that
    \begin{equation}
        \label{E:cond2}
        w_0\in {\operatornamewithlimits{RB}}_{p_0}.
    \end{equation}
    Then, for every $f$ satisfying $f^*\in A$ and every $t\in(0,\infty)$, one has
    \begin{equation}
        \label{E:K-ineq}
        K(f,t;S^{p_0}(w_0),S^{p_1}(w_1))
        \approx
        K(Tf^*,t;\Lambda^{p_0}(\widetilde{w_0}),{\Lambda^{p_1}(\widetilde{w_1})}),
    \end{equation}
    in which the notation from~\eqref{E:tilde-w} has been adopted.  
\end{theorem}

In our next result we characterize the $K$-functional of a pair of fairly general r.i.~quasi-Banach function spaces. For a similar result, in a slightly different setting, see~\cite[Theorem~1a]{Mal:84}, and for some related results in certain more specific situations, see e.g.~\cite{Hol:70,Zip:71,Mil:79,Tor:79,Mil:82,Ara:83} and the references therein. Later we shall apply it to a pair of spaces of type $\Lambda$, and, via~\eqref{E:K-ineq}, to a pair of spaces of type $S$. However, we formulate the theorem for abstract spaces, since such a general statement is of independent interest, and the field of applications might be broader. 

\begin{theorem}
      \label{T:GeneralK}
  Let $A_0$ and $A_1$ be r.i.~quasi-Banach function spaces that admit, respectively, the representation spaces $\overline{A}_0$ and $\overline{A}_1$ over $(0,\infty)$. 
  For $i\in\{0,1\}$, let $\varphi_i(t)=\|\chi_{(0,t)}\|_{\overline{A}_i}$ for $t\in(0,\infty)$, and set
    \begin{equation*}
        \sigma
        =
        \frac{\varphi_0}{\varphi_1}.
    \end{equation*}
    Suppose that
    \begin{equation}
        \label{E:GenProp1}
        \left\|\chi_{(0,t)}(s)g^*(\tfrac{s}{2})\right\|_{\overline{A}_0}
        \lesssim 
        \sigma(t)\left\|g^*\right\|_{\overline{A}_1}\quad\text{for $t\in(0,\infty)$ and $g\in A_1$,}
    \end{equation}
        and
    \begin{equation}
        \label{E:GenProp2}
        \sigma(t)\left\|\chi_{(t,\infty)}(s)h^*(\tfrac{s}{2})\right\|_{\overline{A}_1}
        \lesssim 
        \left\|h^*\right\|_{\overline{A}_0}\quad\text{for $t\in(0,\infty)$ and $h\in A_0$}
    \end{equation}
with constants in the relations `$\lesssim$' independent of $g,\,h$ and $t$. Then
    \begin{equation}
    \label{E:GeneralK}
        K(f,\sigma(t);A_0,A_1)
        \approx
        \left\|\chi_{(0,t)}f^*\right\|_{\overline{A}_0}
        +
        \sigma(t)\left\|\chi_{(t,\infty)}f^*\right\|_{\overline{A}_1}
    \end{equation}
    for $f\in A_0+A_1$ and $t\in(0,\infty)$ with
    constants of the equivalence independent of  $f$ and $t$.
\end{theorem}

We shall finally formulate an explicit formula for the $K$-functional for a pair of spaces of type $S$. 

\begin{theorem}
    \label{T:2}
    Assume that $p_0,p_1\in(0,\infty)$, and $w_0,w_1$ are weights on $(0,\infty)$.  Define the functions $\psi_0,\psi_1\colon (0,\infty)\to[0,\infty]$ by~\eqref{E:varphi} and suppose that
    \begin{equation}
        \label{E:psi-infinity-at-zero}
        \psi_i(0+)=\infty
        \quad\text{for $i\in\{0,1\}$,}
    \end{equation}
    and that conditions~\eqref{E:cond1} and~\eqref{E:cond2} are satisfied. Denote
    \begin{equation}
        \label{E:sigma}
        \theta(t)=
        \frac{\psi_0(t)}
        {\psi_1(t)}
        \quad\text{for $t\in(0,\infty)$.}
    \end{equation}
    Assume that there exists an $\varepsilon>0$ such that
    \begin{equation}
        \label{E:cond3}
        t
        \mapsto \theta(t) \psi_0(t)^{\varepsilon}
        \quad\text{is equivalent to a non-decreasing function on $(0,\infty)$.}
    \end{equation}
    Then
    \begin{align}
        \label{E:K-for-S}
        &K(f,\theta(t);S^{p_0}(w_0),S^{p_1}(w_1))
            \\
        &\quad\approx
        \left(
        \int_{0}^{t}
        \left(
        f^{**}(s)-f^*(s)
        \right)^{p_0}w_0(s)\dd s
        \right)^{\ipo}+
        \theta(t)
        \left(
        \int_{t}^{\infty}
        \left(
        f^{**}(s)-f^*(s)
        \right)^{p_1}w_1(s)\dd s
        \right)^{\ipon}\nonumber
    \end{align}
    for $f\in S^{p_0}(w_0)+S^{p_1}(w_1)$ and $t\in(0,\infty)$
    with constants in the relation `$\approx$' independent of~$f$ and~$t$.
\end{theorem}

The following corollary constitutes a useful particular case of Theorem~\ref{T:2}.

\begin{corollary}
    \label{C:1}
    Let $p\in(1,\infty)$ and $\alpha\in(0,\infty)$. Then
    \begin{align}
        \label{E:corollary-1}
        K(f,t^{\frac{\alpha}{p}};L^p,S^{p}(t^{-\alpha}))
            &\approx
        \left(
        \int_{0}^{t}
        \left(
        f^{**}(s)-f^*(s)
        \right)^{p}\dd s
        \right)^{\ip}
            \nonumber\\
        &\qquad+
        t^{\frac{\alpha}{p}}
        \left(
        \int_{t}^{\infty}
        \left(
        f^{**}(s)-f^*(s)
        \right)^{p}s^{-\alpha}\dd s
        \right)^{\ip}
        \quad\text{for $t\in(0,\infty)$.}
        \nonumber
    \end{align}    
\end{corollary}

We shall now point out an interesting application of Corollary~\ref{C:1} to an inequality of reverse Marchaud type. Such inequalities are central for example in the theory of approximations. We  shall use the known formula for the $K$-functional for the pair $(V^{k,p}(\R^n),L^p(\R^n))$, see e.g.~\cite[Chapter 5, Inequality~(4.42), p.~341]{Ben:88}. It will restrict us to $p\in(1,n)$. We will need the notion of the \emph{modulus of smoothness}. For $h\in{\mathbb R}^n$, the first difference operator $\Delta_h$ is
defined on a real function $f\colon\rn\to\R$ by
\begin{equation*}
   \Delta_hf(x)=f(x+h)-f(x) 
   \quad\text{for $x\in{\mathbb R^n}$,} 
\end{equation*}
and, for $k\in\mathbb N$, $k\ge2$, the $k$-th order difference operator is defined inductively by
\begin{equation*}
    \Delta_h^{k}f(x)=\Delta_h(\Delta_h^{k-1}f)(x),\quad\text{for $x\in{\mathbb R^n}$.}
\end{equation*}
For $p\in[1,\infty]$, the \emph{$p$-modulus of smoothness of order $k$} of a function $f\in L^p({\mathbb R}^n)$ is defined by
$$
\omega_k(f,t)_p=\sup_{|h|\leq t}\|\Delta_h^kf\|_{L^p({\mathbb R}^n)}\quad\text{for $t\in(0,\infty)$.} 
$$

\begin{corollary}
    \label{C:2}
    Assume that $n\in\N$, $n\ge 2$, and $p\in(1,n)$. Let $k\in\N$ be such that 
    \begin{equation}
        \label{E:conditions-on-k}
            2\le k<\frac{n}{p}+1.
    \end{equation}
    Then
    \begin{align}
        \label{E:C2}
        &\left(\int_{0}^{t}
        \left(f^{**}(s)-f^*(s)\right)^{p}\dd s
        \right)^{\frac{1}{p}}
        +
        t^{\frac{k}{n}}
        \left(\int_{t}^{\infty}
        \left(f^{**}(s)-f^*(s)\right)^{p}s^{-\frac{kp}{n}}\dd s
        \right)^{\frac{1}{p}}
            \ls 
        \omega_k(f,t^{\frac{1}{n}})_p
    \end{align}
    for $t\in(0,\infty)$ and every $(k-1)$-times weakly differentiable function $f\colon\rn\to\R$ such that
    \begin{equation}
        \label{E:condition-at-infty}
        \lim_{t\to\infty}|\nabla^{j}f|^{**}(t)=0
        \quad\text{for each $j\in\{0,\dots,k-1\}$.}
    \end{equation}
\end{corollary}

\section{Preliminaries}\label{S:preliminaries}

In this section, we shall collect definitions and other background material and fix notation. 

Let $(\RR,\mu)$ be a~nonatomic $\sigma$-finite measure space with $\mu(\RR)\in(0,\infty]$.
We denote by $\MM(\RR,\mu)$ the set of all $\mu$-measurable functions on $\RR$ whose values lie
in $[-\infty,\infty]$ and by $\MM_+(\RR,\mu)$ the set of all functions in $\MM(\RR,\mu)$ whose values lie in $[0,\infty]$. If there is no risk of confusion, we write only $\MM$ and $\MM_+$, respectively.

For $f\in\MM(\RR,\mu)$, the function $f^*\colon [0,\infty)\to[0,\infty]$, defined by
\begin{equation*}
    f^{\ast}(t)=\inf\{\lambda\geq 0:\mu(\{x\in \RR:\abs{f(x)}>\lambda\})\leq t\}\quad\text{for $t\in[0,\infty)$,}
\end{equation*}
is called the \emph{non-increasing rearrangement} of $f$. The function $f^{**}\colon(0,\infty)\rightarrow[0,\infty]$, defined by
$$
f^{**}(t)=\frac{1}{t}\int_0^t f^*(s)\,d s\quad\text{for $t\in(0,\infty)$},
$$
is called the \emph{maximal non-increasing rearrangement} of $f$. The maximal non-increasing rearrangement is subadditive in the sense that
\begin{equation}
    \label{E:subadditivity-twostar}
    (f+g)^{**}
    \le
    f^{**}+g^{**}
    \quad
    \text{for $f,g\in\MRM$.}
\end{equation}
The operation of taking the non-increasing rearrangement, on the other hand, is not subadditive, but it satisfies merely the estimate 
\begin{equation}
    \label{E:subadditivity-onestar}
    (f+g)^{*}(s+t)
    \le
    f^{*}(s)+g^{*}(t)
    \quad
    \text{for $f,g\in\MRM$ and $s,t\in(0,\infty)$.}
\end{equation}
Two functions $f$ and $g$, notably defined on possibly different measure spaces, are called \emph{equimeasurable}, and denoted $f\sim g$, if $f^*=g^*$ on $(0,\infty)$.
   
 Since our measure space is assumed nonatomic, it is automatically \textit{resonant} in the sense that 
\begin{equation*}
    \int_{0}^{\infty}f^*(t)g^*(t)\,dt
    =
    \sup_{\tilde g\sim g}
    \int_{\mathcal R}|f(x)\tilde g(x)|d\mu(x),
\end{equation*}
in which the supremum is extended over all functions $\tilde g$ equimeasurable with $g$.

For a $\mu$-measurable subset $E$, we shall denote by $\chi_E$ the characteristic function of $E$. 

\begin{definition}
	Let $\lVert \cdot \rVert\colon \mathcal{M}(\mathcal{R}, \mu) \rightarrow [0, \infty]$ be a mapping satisfying $\lVert \, \lvert f \rvert \, \rVert = \lVert f \rVert$ for $f \in \mathcal{M}$. We say that $\lVert \cdot \rVert$ is a \emph{Banach function norm} if its restriction to $\mathcal{M}_+$ satisfies the following axioms:
	\begin{enumerate}[label=\textup{(P\arabic*)}, series=P]
		\item \label{P1} it is a norm, in the sense that it obeys the following three conditions:
		\begin{enumerate}[ref=(\theenumii)]
			\item \label{P1a} it is positively homogeneous, i.e.\ $\forall a \in \mathbb{R} \; \forall f \in \mathcal{M}_+ : \lVert a f \rVert = \lvert a \rvert \lVert f \rVert$,
			\item \label{P1b} it satisfies $\lVert f \rVert = 0 \Leftrightarrow f = 0$  $\mu$-a.e.,
			\item \label{P1c} it is subadditive, i.e.\ $\forall f,g \in \mathcal{M}_+ \: : \: \lVert f+g \rVert \leq \lVert f \rVert + \lVert g \rVert$,
		\end{enumerate}
		\item \label{P2} it has the lattice property, i.e.\ if some $f, g \in \mathcal{M}_+$ satisfy $f \leq g$ $\mu$-a.e., then also $\lVert f \rVert \leq \lVert g \rVert$,
		\item \label{P3} it has the Fatou property, i.e.\ if  some $f_n, f \in \mathcal{M}_+$ satisfy $f_n \uparrow f$ $\mu$-a.e., then also $\lVert f_n \rVert \uparrow \lVert f \rVert $,
		\item \label{P4} $\lVert \chi_E \rVert < \infty$ for $E \subseteq \mathcal{R}$ satisfying $\mu(E) < \infty$,
		\item \label{P5} for every $E \subseteq \mathcal{R}$ satisfying $\mu(E) < \infty$ there exists some positive constant $C_E$, depending only on $E$, such that the inequality $ \int_E f \: d\mu \leq C_E \lVert f \rVert $ is true for $f \in \mathcal{M}_+$.
	\end{enumerate} 
\end{definition}

While the theory of r.i.~spaces is classical and well developed (\cite{Ben:88,Kre:82}), the analogous knowledge about r.i.~quasinorms and corresponding spaces is rather sparse. On the other hand, it is becoming fashionable recently (\cite{Pes:22,Mus:25,Nek:24,Pes:25}). The interest in it stems from the fact that the normed spaces do not hold answers to all challenging contemporary problems. 

\begin{definition}
	Let $\lVert \cdot \rVert\colon \mathcal{M}(\mathcal{R}, \mu) \rightarrow [0, \infty]$ be a mapping satisfying $\lVert \, \lvert f \rvert \, \rVert = \lVert f \rVert$ for $f \in \mathcal{M}$. We say that $\lVert \cdot \rVert$ is a \emph{quasi-Banach function norm} if its restriction to $\mathcal{M}_+$ satisfies the axioms \ref{P2}, \ref{P3} and \ref{P4} of a Banach function norm together with a weaker version of axiom \ref{P1}, namely
	\begin{enumerate}[label=\textup{(Q\arabic*)}]
		\item \label{Q1} it is a quasinorm, in the sense that it satisfies the following three conditions:
		\begin{enumerate}[ref=(\theenumii)]
			\item \label{Q1a} it is positively homogeneous, i.e.\ $\forall a \in \mathbb{R} \; \forall f \in \mathcal{M}_+ : \lVert af \rVert = \lvert a \rvert \lVert f \rVert$,
			\item \label{Q1b} it satisfies  $\lVert f \rVert = 0 \Leftrightarrow f = 0$ $\mu$-a.e.,
			\item \label{Q1c} there is a positive constant $C$ such that
			\begin{equation*}
				\forall f,g \in \mathcal{M}_+ : \lVert f+g \rVert \leq C(\lVert f \rVert + \lVert g \rVert).
			\end{equation*}
		\end{enumerate}
	\end{enumerate}
\end{definition}

A quasi-Banach function norm might, or might not, satisfy (P5). Nontrivial examples of both situations can be obtained by considering the functional
$\|\cdot\|_{L^{1,q}}$ given by
\begin{equation*}
    \|f\|_{L^{1,q}}=\|f^*(t)t^{1-\frac1q}\|_{L^q(0,\infty)}
    \quad\text{for $q\in(0,\infty]$.}
\end{equation*}
Then $\|\cdot\|_{L^{1,q}}$ is a quasinorm for every $q\in(0,\infty]$, it is equivalently normable if and only if $q=1$, and it satisfies  (P5) if and only if $q\in(0,1]$. For further details, see~\cite{Pes:22,Mus:25,Nek:24,Pes:25}.

\begin{definition}
	Let $\lVert \cdot \rVert_X$ be a (quasi-)Banach function norm. We say that $\lVert \cdot \rVert_X$ is \emph{rearrangement\nobreakdash-invariant}, abbreviated \emph{r.i.}, if $\lVert f\rVert_X = \lVert g \rVert_X$ whenever $f, g \in \mathcal{M}$ are equimeasurable.
\end{definition}

\begin{definition}
	Let $\lVert \cdot \rVert_X$ be a (quasi-)Banach function norm. We then define the corresponding \emph{(quasi-)Banach function space} $X$ as the set
	\begin{equation*}
		X = \left \{ f \in \mathcal{M};  \; \lVert f \rVert_X < \infty \right \}.
	\end{equation*}
	
	Furthermore, we will say that $X$ is rearrangement-invariant whenever $\lVert \cdot \rVert_X$ is.
\end{definition}

If
$\lVert \cdot \rVert_X$ is an r.i.~quasi-Banach function norm on $\mathcal{M}(\mathcal{R},\mu)$, then there exists a uniquely determined r.i.~quasi-Banach function norm $\lVert \cdot \rVert_{\overline{X}}$ on $\mathcal{M}([0,\infty), \lambda)$, in which $\lambda$ is the one-dimensional Lebesgue measure, such that, for every $f \in \mathcal{M}(\mathcal{R},\mu)$, one has $\lVert f \rVert_X=\lVert f^* \rVert_{\overline{X}}$. When $\|\cdot\|_X$ is a norm, this is the classical Luxemburg representation theorem (see e.g.~\cite[Chapter~2, Theorem~4.10]{Ben:88}). In the case when $\|\cdot\|_X$ is merely a quasinorm, this is the recent result of \cite[Theorem~3.1]{Mus:25}. 

\begin{definition}
	Assume that $\lVert \cdot \rVert_X$ is an r.i.~quasi-Banach function norm and $X$ is the corresponding r.i.~quasi-Banach function space. We then define the \emph{fundamental function} $\varphi_X$ of $X$ by
	\begin{equation*}
		\varphi_X(t) = \lVert \chi_{E_t} \rVert_{X}
        \quad\text{for $t\in[0,\mu(\RR))$,}
	\end{equation*}
where $E_t$ is any subset of $\mathcal{R}$ with $\mu(E_t) = t$. 
\end{definition}

Note that the rearrangement invariance of $X$ guarantees that its fundamental function is well defined. Moreover, the fact that $\mu$ is nonatomic implies that $\varphi_X$ is defined on the entire interval $t\in[0,\mu(\RR))$.

\begin{remark}
    Observe that if $X$ allows a representation space $\overline{X}$, then
	\begin{equation*}
		\varphi_X(t) = \lVert \chi_{(0,t)} \rVert_{\overline{X}}
        \quad\text{for $t\in[0,\mu(\RR))$.}
	\end{equation*}
    
\end{remark}


\begin{definition}
    \label{D:classical-lorentz-spaces}
    Fix $p\in(0,\infty]$ and a weight $w$ on $(0,\infty)$. We define the functionals
    \begin{align}
        &\|f\|_{\Lambda^p(w)}
        =
        \begin{cases}
            \|f^*w^\frac{1}{p}\|_{L^p(0,\infty)}
                &\text{if $p\in(0,\infty)$}
                    \\
            \sup\limits_{t\in(0,\infty)}f^*(t)w(t)
                &\text{if $p=\infty$,}
        \end{cases}
            \\
        &\|f\|_{\Gamma^p(w)}
        =
        \begin{cases}
            \|f^{**}w^\frac{1}{p}\|_{L^p(0,\infty)}
                &\text{if $p\in(0,\infty)$}
                    \\
            \sup\limits_{t\in(0,\infty)}f^{**}(t)w(t)
                &\text{if $p=\infty$,}
        \end{cases}
            \intertext{and}
        &\|f\|_{S^p(w)}
        =
        \begin{cases}
        \|(f^{**}-f^*)w^\frac{1}{p}\|_{L^p(0,\infty)}
                &\text{if $p\in(0,\infty)$}
                    \\
            \sup\limits_{t\in(0,\infty)}(f^{**}(t)-f^*(t))w(t)
                &\text{if $p=\infty$}
        \end{cases}
    \end{align}
    for every $f\in\MM(\RR,\mu)$, in which $\|\cdot\|_{L^p(0,\infty)}$ is the usual Lebesgue (quasi-)norm on $(0,\infty)$. Then the \emphdef{classical Lorentz spaces} $\Lambda^p(w)$ and $\Gamma^p(w)$ are defined as the sets of all functions $f\in\MM(\RR,\mu)$ for which the respective functionals are finite. We define the space $S^p(w)$ as the collection of all measurable functions $f$ having $\|f\|_{S^p(w)}<\infty$ and satisfying $\lim_{t\to\infty}f^*(t)=0$. 
\end{definition}

\begin{remark}
    The space $S^p(w)$ is defined slightly differently than its companions of types $\Lambda$ and $\Gamma$ because the functional $f^{**}-f^*$ vanishes on constant functions, which obviously calls for some normalization. Unlike the spaces of the other two types, spaces of type $S$ do not measure the \emph{size} of a function, but rather its \emph{smoothness} through oscillation. 
\end{remark}

\begin{remark}
    We note that classical Lorentz spaces of all types have their natural representation spaces, regardless whether they are r.i.~quasi-Banach function spaces or not, due to the very nature of their governing functionals. More precisely, we have (appealing once again to Convention~\ref{CON:representation})
    \begin{equation}
        \label{E:representation-classical}
        \overline
        {\Lambda^{p}(w)}
        =
        {\Lambda^{p}(w)},
        \quad
        \overline
        {\Gamma^{p}(w)}
        =
        {\Gamma^{p}(w)},
        \quad
        \overline
        {S^{p}(w)}
        =
        {S^{p}(w)}.
    \end{equation}
    We will be using these relations without further warning.
\end{remark}

In the proofs of the main results and, in particular, in the applications, we will need some functional properties of spaces of type $\Lambda$ which we will recall now. In~\cite[Theorem~1.1]{Spa:07} 
(for other variants see also~\cite[Corollary~2.2]{Car:93}, \cite[page 6]{Kam:04} or~\cite{Cwi:04}) it was observed that the functional $\|\cdot\|_{\Lambda^p(w)}$ for $p\in(0,\infty)$ is a quasinorm if and only if the function $W$ (the primitive function of $w$) satisfies the $\Delta_2$-condition, i.e., there is a positive constant $C$ such that
\begin{equation}
    \label{E:Delta-2-for-W}
    W(2t)\le CW(t)\quad\text{for $t\in(0,\infty)$.}
\end{equation}
Under the assumption~\eqref{E:Delta-2-for-W}, the space $\Lambda^p(w)$ is, for any $p\in(0,\infty)$, a~quasi-Banach space. In particular, it is complete and has the Fatou property. If, moreover, $w$ is integrable near zero, then the functional $\|\cdot\|_{\Lambda^p(w)}$ satisfies the axiom (P4), resulting in the space $\Lambda^p(w)$ being an~r.i.~quasi-Banach function space.
Furthermore, thanks to the recent advances of~\cite[Theorem~3.23]{Nek:24}, for every fixed $a\in(0,\infty)$, the \textit{dilation operator} $D_a$, defined by 
\begin{equation*}
    D_ag(t)=g(at)
    \quad\text{for $g\in\M(0,\infty)$ and $t\in(0,\infty)$,}
\end{equation*}
is bounded on the representation space of $\Lambda^{p}(w)$ in the sense that there exists a positive constant $C$ such that, retiring to Convention~\ref{CON:representation},
\begin{equation}
    \label{E:dilation}
    \|g(at)\|_{{\Lambda^{p}(w)}}
    \le C
    \|g(t)\|_{{\Lambda^{p}(w)}}
    \quad\text{for $g\in\M(0,\infty)$,}
\end{equation}
in which $C$ may depend on $p,w,a$, but not $g$.

We will finish this section by recalling the crucial notion of the $K$-functional and its important modification.

\begin{definition}
    \label{D:K-functional}
    Let $(X_0,X_1)$ be a pair of quasi-Banach function spaces over a common measure space $\RM$. Let $X_0+X_1$ denote the sum of spaces consisting of all functions $f\in \M\RM$ such that there exists a decomposition $f=f_0+f_1$ with $f_i\in X_i$, $i\in\{0,1\}$. Then, for every $f\in X_0+X_1$ and every $t\in[0,\infty)$, we define the \emph{$K$-functional} by
    \begin{equation}
        \label{E:K-functional}
        K(f,t;X_0,X_1) = \inf \left\{ \|f_0\|_{X_0}+t\|f_1\|_{X_1}: f=f_0+f_1,\ f_i\in X_i,\
        i\in\{0,1\}
        \right\}.
    \end{equation}
\end{definition}

\begin{remark}
    Given $p_0,p_1\in(0,\infty)$ and a couple of weights $w_0,w_1$, we define the $K$-functional for the pair $(S^{p_0}(w_0),S^{p_1}(w_1))$, and, indeed, also of the pair $(\Lambda^{p_0}(w_0),\Lambda^{p_1}(w_1))$, over a common measure space, in the same way as in Definition~\ref{D:K-functional}, without a-priori knowledge whether either of the corresponding functionals is a quasinorm.
\end{remark}

We will also need the restricted version of the $K$-functional, defined for non-increasing functions (for further details, see~\cite{Cer:96}). 
For a quasi-Banach function space $X$ over the measure space $[0,\infty)$ endowed with the one-dimensional Lebesgue measure, we denote by $X^d$ the cone of all non-increasing functions contained in $X$.

\begin{definition}
    Let $(X_0,X_1)$ be a pair of quasi-Banach function spaces over the measure space $[0,\infty)$ endowed with the one-dimensional Lebesgue measure. For every $f\in X_0^d+X_1^d$ and every $t\in[0,\infty)$, we define the \emph{$K^d$-functional} by
    \begin{align}
        \label{E:K-d-functional}
        K^d(f,t;X_0,X_1)
        &=K(f,t;X_0^d,X_1^d) 
            \\
        &= \inf \left\{ \|f_0\|_{X_0}+t\|f_1\|_{X_1}: f=f_0+f_1,\ f_i\in X_i^d,\ i\in\{0,1\}
        \right\}
        .\nonumber
    \end{align}
\end{definition}
Observe that one always has
\begin{equation}
    \label{E:inequality-of-k-functionals}
    K(f,t;X_0,X_1)\le K^d(f,t;X_0,X_1)
\end{equation}
for every admissible $f$ and $t$, since the functional on the right extends the infimum over a smaller set than that on the left side. 
On the other hand, the question whether or not the converse inequality to~\eqref{E:inequality-of-k-functionals} holds is highly nontrivial and of great interest. In this direction, a very useful tool is the \emph{decomposition lemma} (\cite[Lemma~1]{Cer:96}) which states that if three right continuous non-increasing functions $f,g,h$ satisfy $f\le g+h$, then there are non-increasing functions $f_0,f_1$ such that $f_0\le g$, $f_1\le h$ and $f=f_0+f_1$.

\section{Proofs}

We begin by recalling a useful elementary relation which establishes a bridge between the $\Lambda$ and the $S$ worlds, and which will come to play several times in our proofs. A part of it, without proof, was mentioned in the course of the proof of~\cite[Theorem~3.3]{Car:05}. For the sake of self-containment, we include its simple verification.

\begin{lemma}\label{L:SandTLambda}
Let $p\in(0,\infty)$ and let $w$ be a weight on $(0,\infty)$. Then, with $T$ from~\eqref{E:T} and $\widetilde w$ from~\eqref{E:tilde-w}, we have, for every $f\in S^{p}(w)$ and every $t\in(0,\infty)$,
\begin{align}
    \label{E:prequel-S-Lambda}
    \left(\int_0^t\left(f^{**}(s)-f^*(s)\right)^pw(s)\dd s\right)^{\frac1p}
    =
    \left(
    \int_{\frac{1}{t}}^\infty Tf^*(s)^p
    \widetilde w(s)\dd s\right)^{\frac1p}
    \intertext{and}
    \label{E:prequel-S-Lambda-dual}
    \left(\int_t^\infty\left(f^{**}(s)-f^*(s)\right)^pw(s)\dd s\right)^{\frac1p}
    =
    \left(
    \int_0^{\frac{1}{t}}Tf^*(s)^p
    \widetilde w(s)
    \dd s\right)^{\frac1p}. 
\end{align}
In particular, 
    \begin{equation}
        \label{E:S-Lambda}
        \|f\|_{S^{p}(w)}
        =
        \|Tf^*\|_{\Lambda^{p}(\tilde w)}.
    \end{equation}
\end{lemma}

\begin{proof}
    Fix $t\in(0,\infty)$. Changing variables $s\mapsto\frac{1}{s}$ and using definitions of $T$ and $\tilde w$, we arrive at
    \begin{align*}
    		\left(\int_0^t\left(f^{**}(s)-f^*(s)\right)^pw(s)\dd s\right)^{\frac1p}
            &=
    		\left(\int_{\frac{1}{t}}^\infty\left(f^{**}(\tfrac1s)-f^*(\tfrac1s)\right)^pw(\tfrac1s)\frac{\dd s}{s^2}\right)^{\frac1p}
                \\
    		&=
    		\left(\int_{\frac{1}{t}}^\infty\left(sTf^*(s)\right)^pw(\tfrac1s)\frac{\dd s}{s^2}\right)^{\frac1p}
                \\
    		&=
    		\left(
            \int_{\frac{1}{t}}^\infty
            Tf^*(s)^p\widetilde w(s)\dd s\right)^{\frac1p},
    \end{align*}
    since $Tf^*$ is non-increasing (cf.~\eqref{E:T-star-alternative}). This establishes~\eqref{E:prequel-S-Lambda}. The relation~\eqref{E:prequel-S-Lambda-dual} can be proved analogously. Finally,~\eqref{E:S-Lambda} follows on letting $t\to\infty$ in~\eqref{E:prequel-S-Lambda}.
\end{proof}

Let us still recall that, for an integrable function $f\colon \rn\to\R$ and $0<t<s<\infty$, integration by parts yields
    \begin{equation}
        \label{E:formula-by-parts}
        g^{**}(t)-g^{**}(s)
        =
        \int_{t}^{s} 
        \frac{g^{**}(\tau)-g^*(\tau)}{\tau}\dd \tau
    \end{equation}
(see~\cite[Chapter~5, Formula~(7.11), page~379]{Ben:88}). If, moreover $\lim_{s\to\infty}g^{**}(s)=0$, then, on taking $s\to\infty$ in~\eqref{E:formula-by-parts}, one gets
(\cite[Chapter~5, Formula~(7.29), page~384]{Ben:88} or~\cite[Formula~(2.7)]{Kol:07})
    \begin{equation}
        \label{E:formula-rearrangement}
        g^{**}(t)
        =
        \int_{t}^{\infty} 
        \frac{g^{**}(s)-g^*(s)}{s}\dd s
        \quad\text{for $t\in(0,\infty)$.}
    \end{equation}
    We shall use these observations in the forthcoming proof.

\begin{proof}[Proof of Theorem~\ref{T:1.1}]
    Consider a function
   $f\in S^{p_0}(w_0)+S^{p_1}(w_1)$. Then, arguing analogously to the proof of~\cite[Proposition~3.1]{Cer:03} and employing the measure preserving
    transformation, whose existence is guaranteed  by Ryff’s theorem (\cite[Chapter~2, Theorem~7.5 and Corollary~7.6]{Ben:88}), we get that 
    \begin{equation*}
        f^*\in S^{p_0}(w_0)^d+S^{p_1}(w_1)^d
    \end{equation*}
    (note that here Convention~\ref{CON:representation} applies) and, by~\eqref{E:inequality-of-k-functionals},
   \begin{align}
        \label{E:K-less-than-Kd}
       K(f,t;S^{p_0}(w_0),S^{p_1}(w_1))
       \le
        K^d(f^*,t;S^{p_0}(w_0),S^{p_1}(w_1)).
   \end{align}
   Observe that, for every pair $f_0$, $f_1$ of nonnegative non-increasing functions on $(0,\infty)$ such that $f^*=f_0+f_1$, the function $Tf^*$ is non-increasing  (cf.~\eqref{E:T-star-alternative}), and one has 
	\begin{equation}
		Tf^*(t)
        =Tf_0(t)
        +Tf_1(t).
	\end{equation}
   Therefore, using Lemma~\ref{L:SandTLambda},
    we obtain
	\begin{align}
		&K^d(f^*,t;S^{p_0}(w_0),S^{p_1}(w_1))
        \label{E:K-inequality-first}
            \\
		&\qquad=
		\inf\left\{\|f_0\|_{S^{p_0}(w_0)}+t\|f_1\|_{S^{p_1}(w_1)}; f^*=f_0+f_1, f_i\in S^{p_i}(w_i)^d, i\in\{0,1\}\right\}     \nonumber\\
		&\qquad=
		\inf\left\{\|Tf_0\|_{\Lambda^{p_0}(\widetilde{w_0})}+t\|Tf_1\|_{\Lambda^{p_1}(\widetilde{w_1})}; f^*=
        f_0+f_1,f_i\in S^{p_i}(w_i)^d, i\in\{0,1\} \right\}
            \nonumber\\
		&\qquad\geq
		\inf\left\{\|g_0\|_{\Lambda^{p_0}(\widetilde{w_0})}+t\|g_1\|_{\Lambda^{p_1}(\widetilde{w_1})}; Tf^*=g_0+g_1, g_i\in \Lambda^{p_i}(\widetilde{w_i})^d\right\}
            \nonumber\\
		&\qquad=
		K^d(Tf^*,t;\Lambda^{p_0}(\widetilde{w_0}),\Lambda^{p_1}(\widetilde{w_1})).
        \nonumber
	\end{align}		
    We shall now show a converse inequality to~\eqref{E:K-inequality-first}. First note that it follows from~\eqref{E:K-inequality-first} that $Tf^*$ is a non-increasing element of $\Lambda^{p_0}(\widetilde{w_0})+\Lambda^{p_1}(\widetilde{w_1})$. Let $g_0\in \Lambda^{p_0}(\widetilde{w_0}) and g_1\in \Lambda^{p_1}(\widetilde{w_1})$ be non-increasing functions such that $Tf^*=g_0+g_1$. Set $h_0=Tg_0$ and $h_1=Tg_1$. Then $h_i\in S^{p_i}(w_i)^d$ for $i\in\{0,1\}$, in particular, $h_0$ and $h_1$ are non-increasing, and, applying \eqref{E:T-idempotent}, we have
	\begin{equation*}
		Th_0+Th_1
        =TTg_0+TTg_1
        =g_0+g_1=Tf^*.
	\end{equation*}
    Therefore, by Lemma~\ref{L:SandTLambda},
	\begin{align}
		&\|g_0\|_{\Lambda^{p_0}(\widetilde{w_0})}+t\|g_1\|_{\Lambda^{p_1}(\widetilde{w_1})}\label{E:lower}
            \\
		&\qquad=
		\|Th_0\|_{\Lambda^{p_0}(\widetilde{w_0})}+t\|Th_1\|_{\Lambda^{p_1}(\widetilde{w_1})}
		=
		\|h_0\|_{S^{p_0}(w_0)}+t\|h_1\|_{S^{p_1}(w_1)}\nonumber
            \\
		&\qquad\geq
		\inf\left\{\|f_0\|_{S^{p_0}(w_0)}+t\|f_1\|_{S^{p_1}(w_1)}; Tf^*=Tf_0+Tf_1, f_i\in S^{p_i}(w_i)^d, i\in\{0,1\}\right\}\nonumber
            \\
		&\qquad=
		\inf\left\{\|f_0\|_{S^{p_0}(w_0)}+t\|f_1\|_{S^{p_1}(w_1)}; f^*=f_0+f_1, f_i\in S^{p_i}(w_i)^d, i\in\{0,1\}\right\} \nonumber
            \\
        &\qquad
        =K^d(f^*,t;S^{p_0}(w_0),S^{p_1}(w_1)).\nonumber
	\end{align}
    As $g_0\in \Lambda^{p_0}(\widetilde{w_0}), g_1\in \Lambda^{p_1}(\widetilde{w_1})$ were chosen arbitrarily, taking the infimum over all such decompositions on the left-hand side of~\eqref{E:lower} yields
	\begin{align}
    \label{E:K-inequality-second}
		K^d(Tf^*,t;\Lambda^{p_0}(\widetilde{w_0}),\Lambda^{p_1}(\widetilde{w_1}))
		\geq
		K^d(f^*,t;S^{p_0}(w_0),S^{p_1}(w_1)).
	\end{align}
    Combining~\eqref{E:K-inequality-first} and~\eqref{E:K-inequality-second}, we arrive at
	\begin{equation}
    \label{E:Kd=Kd}
		K^d(f^*,t;
        S^{p_0}(w_0),S^{p_1}(w_1))
		=
		K^d(Tf^*,t;\Lambda^{p_0}(\widetilde{w_0}),\Lambda^{p_1}(\widetilde{w_1})).
	\end{equation}
    By \eqref{E:cond1}, we have that
	\begin{equation*}
        \left(\int_{t}^{\infty}s^{-p_i}w_i(s)\dd s
        \right)^{\frac{1}{p_i}}
        \le C
        \left(\int_{2t}^{\infty}s^{-p_i}w_i(s)\dd s
        \right)^{\frac{1}{p_i}}
        \quad\text{for $i\in\{0,1\}$ and $t\in(0,\infty)$.}
    \end{equation*}
    Changing variables $t\mapsto \frac{1}{t}$, we obtain
		\begin{equation*}
        \left(\int_{0}^{1/t}\widetilde{w_i}(s)\dd s
        \right)^{\frac{1}{p_i}}
        \le C
        \left(\int_{0}^{1/2t}\widetilde{w_i}(s)\dd s
        \right)^{\frac{1}{p_i}}
    \quad\text{for $i\in\{0,1\}$ and $t\in(0,\infty)$,}    
    \end{equation*}
    and thus	
    \begin{equation}
    \label{E:wtilde2}
        \left(\int_{0}^{2t}\widetilde{w_i}(s)\dd s
        \right)^{\frac{1}{p_i}}
        \le C
        \left(\int_{0}^{t}\widetilde{w_i}(s)\dd s
        \right)^{\frac{1}{p_i}}
        \quad\text{for $i\in\{0,1\}$ and $t\in(0,\infty)$.}
	\end{equation}
    Owing to~\cite[Corollary~2.2]{Car:93}, the inequality~\eqref{E:wtilde2} guarantees that  $\Lambda^{p_0}(\widetilde{w_0})$ and $\Lambda^{p_1}(\widetilde{w_1})$ are~quasi-Banach spaces. Thanks to this fact, we can employ the relation proved at the end of the proof of~\cite[Proposition~3.1]{Cer:03} (note that their assumption~\cite[Inequality~(2.2)]{Cer:03}  is our~\eqref{E:wtilde2}), and obtain thereby
    \begin{equation}
    \label{E:Kd=K}
		K^d(Tf^*,t;\Lambda^{p_0}(\widetilde{w_0}),\Lambda^{p_1}(\widetilde{w_1}))
		\approx
		K(Tf^*,t;\Lambda^{p_0}(\widetilde{w_0}),\Lambda^{p_1}(\widetilde{w_1})).
	\end{equation}
    Combining 
    \eqref{E:K-less-than-Kd}, \eqref{E:Kd=Kd}, and \eqref{E:Kd=K}, we get
\begin{equation}\label{E:K<K}
		K(f,t;S^{p_0}(w_0),S^{p_1}(w_1))
		\lesssim
		K(Tf^*,t;\Lambda^{p_0}(\widetilde{w_0}),\Lambda^{p_1}(\widetilde{w_1})).
	\end{equation}
		
    To establish the inequality converse to~\eqref{E:K<K}, fix $f$ in $S^{p_0}(w_0)+S^{p_1}(w_1)$ and let $f_0\in S^{p_0}(w_0)$, $f_1\in S^{p_1}(w_1)$ be such that $f=f_0+f_1$.
    Then, by the subadditivity of the maximal non-increasing rearrangement~\eqref{E:subadditivity-twostar}, and the  estimate $f_1^*(2t)\leq f^*(t)+f_0^*(t)$, which follows from~\eqref{E:subadditivity-onestar} as a special case with $s=t$, we have that
	\begin{align*}
		&f^{**}(t)-f^*(t)
		\leq
		f_0^{**}(t)+f_1^{**}(t)-f_1^*(2t)+f_0^*(t)
		\leq
		2f_0^{**}(t)+f_1^{**}(t)-f_1^*(2t)
            \\
		&\qquad=
		2f_0^{**}(t)+\frac1t\int_0^{2t}f_1^*(s)\dd s-\frac1t\int_t^{2t}f_1^*(s)\dd s-f_1^*(2t)
        \quad\text{for $t\in(0,\infty)$.}
    \end{align*}
    Since, by monotonicity, 
    \begin{equation*}
        f_1^*(2t)
        \le\frac{1}{t}
        \int_{t}^{2t}f_1^*(s)\,ds,
    \end{equation*}
    this implies
    \begin{align*}
		&f^{**}(t)-f^*(t)
		\leq    
        2f_0^{**}(t)+2(f_1^{**}(2t)-f_1^*(2t))
        \quad\text{for $t\in(0,\infty)$.}
	\end{align*}
    Applying this to $\frac1t$ in place of $t$ and using~\eqref{E:T-star}, we get
	\begin{equation*}
		tTf^*(t)
		\leq
		2f_0^{**}(\tfrac1t)+2(f_1^{**}(\tfrac2t)-f_1^*(\tfrac2t))
		=
		2f_0^{**}(\tfrac1t)+tTf_1^*(\tfrac{t}2)
        \quad\text{for $t\in(0,\infty)$.}
	\end{equation*}
    Dividing this by $t$, we obtain the relation
	\begin{equation}
        \label{E:upper-for-Tf-star}
		Tf^*(t)
		\leq
		\tfrac2tf_0^{**}(\tfrac1t)+Tf_1^*(\tfrac{t}2)
        \quad\text{for $t\in(0,\infty)$.}
	\end{equation}
    We claim that
    \begin{equation}
    \label{E:doublingT}
		\|Tf_1^*(\tfrac{t}2)\|_{\Lambda^{p_1}(\widetilde{w_1})}
		\lesssim		
		\|Tf_1^*(t)\|_{\Lambda^{p_1}(\widetilde{w_1})}.
	\end{equation}
    To prove~\eqref{E:doublingT}, we first note that it is equivalent to the inequality
	\begin{equation*}
		\int_0^{\infty}Tf_1^*(s)^{p_1}\widetilde{w_1}(2s)\dd s
		\lesssim		
		\int_0^{\infty}Tf_1^*(s)^{p_1}\widetilde{w_1}(s)\dd s,
	\end{equation*}
    which, since $Tf_1^*$ is non-increasing, would in turn follow from the embedding
    \begin{equation}
        \label{E:lambda-embedding}
        \Lambda^{p_1}(\widetilde{w_1}(s))
        \hookrightarrow
        \Lambda^{p_1}(\widetilde{w_1}(2s))
    \end{equation}
    if the embedding was true. However, in \cite[Proposition 1(a)]{Ste:23}, it is shown that the embedding~\eqref{E:lambda-embedding} holds if and only if
	\begin{equation*}
    \int_0^{t}\widetilde{w_1}(2s)\dd s
		\ls	
		\int_0^{t}\widetilde{w_1}(s)\dd s
        \quad\text{for $t\in(0,\infty)$,}
	\end{equation*}
    which, however, is \eqref{E:wtilde2} in disguise (as is readily verified by changing variables). This establishes~\eqref{E:doublingT}.
		
    By the decomposition lemma and~\eqref{E:upper-for-Tf-star}, we find non-increasing functions $g_0, g_1$ such that 
    \begin{equation}
        \label{E:split-of-T}
        Tf^*(t)
        =g_0(t)+g_1(t)
    \end{equation}
    and 
    \begin{equation}
        \label{E:decomposition}
        g_0(t)\leq \tfrac2tf_0^{**}(\tfrac1t),\quad g_1(t)\leq Tf_1^*(\tfrac{t}2)
        \quad\text{for $t\in(0,\infty)$.}
    \end{equation}
    Applying $T$ to each side of~\eqref{E:split-of-T}, using its linearity and~\eqref{E:T-idempotent}, we get 	
    \begin{equation*}
		f^*(t)
        =Tg_0(t)+Tg_1(t)
        \quad\text{for $t\in(0,\infty)$.}
	\end{equation*}
    Consequently, using~\eqref{E:split-of-T},
    \eqref{E:decomposition},
    \eqref{E:doublingT} and
    \eqref{E:S-Lambda}, we infer that
	\begin{align}
        \label{E:K-by-gamma}
		K(Tf^*,t;\Lambda^{p_0}(\widetilde{w_0}),\Lambda^{p_1}(\widetilde{w_1}))
		&\leq
		\|g_0\|_{\Lambda^{p_0}(\widetilde{w_0})}+t\|g_1\|_{\Lambda^{p_1}(\widetilde{w_1})}
            \\
		&\leq
		\|\tfrac2tf_0^{**}(\tfrac1t)\|_{\Lambda^{p_0}(\widetilde{w_0})}+t\|Tf_1^*(\tfrac{t}2)\|_{\Lambda^{p_1}(\widetilde{w_1})}
            \nonumber\\
		&\lesssim
		2\|f_0^{**}\|_{\Lambda^{p_0}(w_0)}+t\|Tf_1^*\|_{\Lambda^{p_1}(\widetilde{w_1})}
            \nonumber\\
		&=
		2\|f_0\|_{\Gamma^{p_0}(w_0)}+t\|f_1\|_{S^{p_1}(w_1)}
        \quad\text{for $t\in(0,\infty)$.}\nonumber
	\end{align}
    
    We next claim that~\eqref{E:cond2} implies
    \begin{equation}
        \label{E:gamma-s}
        \Gamma^{p_0}(w_0)=S^{p_0}(w_0).
    \end{equation}
    For $p_0\in[1,\infty)$, this follows from~\cite[Corollary~4.3]{Car:08} (for $p_0\in(1,\infty)$ also from~\cite[Theorem~3.3]{Car:05}). Both these sources use duality methods and hence do not allow extension to the full range of $p_0\in(0,\infty)$. We shall give a proof for $p_0\in(0,1)$.
    
    We first note that~\eqref{E:gamma-s} is equivalent to saying that there exists a positive constant $C$ such that
    \begin{equation}
        \label{E:gamma-s-1}
        \int_{0}^{\infty}
        f^{**}(t)^{p_0}
        w_0(t)\,dt
        \le C
        \int_{0}^{\infty}
        \left(f^{**}(t)-f^*(t)\right)^{p_0}
        w_0(t)\,dt
        \quad\text{for $f\in\mathcal M(\mathcal R,\mu)$,}
    \end{equation}
    owing to the fact that the converse inequality, with $C=1$, is obvious. Using  \eqref{E:formula-rearrangement} and~\cite[Corollary~5.3]{Car:08}, we get 
    \begin{align*}
    f^{**}(t)^{p_0}
        &=  
        \left(\int_t^ \infty \frac{f^{**}(s)-f^{*}(s)}{s}ds\right)^{p_0}
        \lesssim
        \int_t^ \infty \frac{(f^{**}(s)-f^{*}(s))^{p_0}}{s}ds
        \quad\text{for $t\in(0,\infty)$.}
   \end{align*}
    Integrating with respect to $w_0(t)\,dt$ and using Fubini's theorem, we have
 \begin{align*}
        \int_{0}^{\infty}
        f^{**}(t)^{p_0}
        w_0(t)\,dt
               &\lesssim \int_{0}^{\infty}
        \int_t^ \infty \frac{(f^{**}(s)-f^{*}(s))^{p_0}}{s}ds\,
        w_0(t)\,dt\\
         &= \int_{0}^{\infty}
         \frac{(f^{**}(s)-f^{*}(s))^{p_0}}{s}\int_0^s  w_0(t)\,dt\, ds.
    \end{align*}
Now, applying~\cite[Corollary~5.3]{Car:08} once again, we obtain that the inequality
    \begin{align*}
        & \int_{0}^{\infty}
         \frac{(f^{**}(s)-f^{*}(s))^{p_0}}{s}\int_0^s  w_0(t)\,dt ds  \le C
        \int_{0}^{\infty}
        \left(f^{**}(t)-f^*(t)\right)^{p_0}
        w_0(t)\,dt
    \end{align*}
    holds with $C$ given by
    \begin{equation*}
        C\approx
        \sup_{t\in(0,\infty)}
        \frac
        {\int_{t}^{\infty}\tau^{-p_0-1}\int_0^{\tau}  w_0(s)\, ds\,d\tau}
        {\int_{t}^{\infty}\tau^{-p_0}  w_0(\tau)\,d\tau}
    \end{equation*}
   as long as this expression is finite. However, using Fubini's theorem, one can easily verify that 
     \begin{equation*}
        C\lesssim 1+
        \sup_{t\in(0,\infty)}
        \frac
        {t^{-p_0}\int_{0}^{t}w_0(s)\, ds}
        {\int_{t}^{\infty}\tau^{-p_0}  w_0(\tau)\,d\tau},
    \end{equation*}
    which is finite owing to~\eqref{E:cond2}. This establishes~\eqref{E:gamma-s-1}, and hence, in turn,~\eqref{E:gamma-s}, for $p_0\in(0,1)$. 
    
    For the sake of completeness, let us mention that an alternative proof of~\eqref{E:gamma-s} for $p_0\in(0,1)$ can be tailored via an approximation argument in which we fix a function $f$ and then we effectively replace $f^*$ by a sequence of functions of the form $\int_t^{\infty}h(s)\,ds$ which converge to $f^*$ monotonically from below, and applying the monotone convergence theorem, owing to a recent result from~\cite[Theorem 2.1, Part (ii)]{Gog:22}. This approach gives works equally fine, albeit through a more lengthy argument. We omit the details.

    Now, combination of~\eqref{E:K-by-gamma} with~\eqref{E:gamma-s} yields
	\begin{equation*}
		K(Tf^*,t;\Lambda^{p_0}(\widetilde{w_0}),\Lambda^{p_1}(\widetilde{w_1}))
		\lesssim
		2\|f_0\|_{S^{p_0}(w_0)}+t\|f_1\|_{S^{p_1}(w_1)}
        \quad\text{for $t\in(0,\infty)$.}
	\end{equation*}
    As functions $f_0$ and $f_1$ were taken arbitrary, we get	
    \begin{equation}
    \label{E:K>K}
		K(Tf^*,t;\Lambda^{p_0}(\widetilde{w_0}),\Lambda^{p_1}(\widetilde{w_1}))
		\lesssim
		K(f,t;S^{p_0}(w_0),S^{p_1}(w_1))
        \quad\text{for $t\in(0,\infty)$.}
	\end{equation}		
    Finally, 
    \eqref{E:K<K} and \eqref{E:K>K} give
	\begin{equation}
    \label{E:K-equiv-K}
		K(f,t;S^{p_0}(w_0),S^{p_1}(w_1))
        \approx
        K(Tf^*,t;\Lambda^{p_0}(\widetilde{w_0}),\Lambda^{p_1}(\widetilde{w_1}))
        \quad\text{for $t\in(0,\infty)$,}
	\end{equation}
    as desired.
\end{proof}

\begin{proof}[Proof of Theorem~\ref{T:GeneralK}]
Fix $t\in(0,\infty)$ and a function $f\in A_0+A_1$, and suppose that $f_0\in A_0$ and $f_1\in A_1$ are such that $f=f_0+f_1$. By~\eqref{E:subadditivity-onestar}, $f^*(s+t)\leq f_0^*(s)+f_1^*(t)$ for every $s,t\in (0,\infty)$. Thus, 
    \begin{align*}
        &\left\|\chi_{(0,t)}f^*\right\|_{\overline{A}_0}
        +
        \sigma(t)\left\|\chi_{(t,\infty)}f^*\right\|_{\overline{A}_1}
            \\
        &\qquad\lesssim
        \left\|\chi_{(0,t)}(s)f_0^*(\tfrac{s}{2})\right\|_{\overline{A}_0}
        +
        \left\|\chi_{(0,t)}(s)f_1^*(\tfrac{s}{2})\right\|_{\overline{A}_0}
            \\
        &\qquad \qquad +
        \sigma(t)\left(
        \left\|\chi_{(t,\infty)}(s)f_0^*(\tfrac{s}{2})\right\|_{\overline{A}_1}
        +
        \left\|\chi_{(t,\infty)}(s)f_1^*(\tfrac{s}{2})\right\|_{\overline{A}_1}
        \right).
    \end{align*}
    Applying \eqref{E:GenProp1} to the first summand and \eqref{E:GenProp2} to the last one,
    we get     
    \begin{equation*}
        \left\|\chi_{(0,t)}f^*\right\|_{\overline{A}_0}
        +
        \sigma(t)\left\|\chi_{(t,\infty)}f^*\right\|_{\overline{A}_1}
        \lesssim
        \left\|f_0^*\right\|_{\overline{A}_0}
        +
        \sigma(t)
        \left\|f_1^*\right\|_{\overline{A}_1}.
    \end{equation*}
Taking the infimum over all representations $f=f_0+f_1$, $f_0\in A_0$, $f_1\in A_1$, we arrive at
\begin{equation}
    \label{E:RHS<LHS}
        \left\|\chi_{(0,t)}f^*\right\|_{\overline{A}_0}
        +
        \sigma(t)\left\|\chi_{(t,\infty)}f^*\right\|_{\overline{A}_1}
        \lesssim
        K(f,\sigma(t);A_0,A_1).
    \end{equation}
To prove the converse estimate, we take a particular decomposition of $f$. Fix $t\in (0, \infty)$. We find $f_0$, $f_1$ so that $f=f_0+f_1$ and
    \begin{align*}
        f^*_0(s)&=(f^*(s)-f^*(t))\chi_{(0,t)}(s)\text{ for $s\in (0,\infty)$},\\
        f^*_1(s)&=
            \left\{
            \begin{array}{ll}
            f^*(t) &\text{for $s\in (0,t]$},\\
            f^*(s) &\text{for $s\in (t,\infty).$}
            \end{array} 
            \right. 
    \end{align*}
Then, using subsequently the definition of $f_0$ and $f_1$, the formula for $\sigma$, the definition of $\varphi_0$, the monotonicity of $f^*$, and properties of quasinorms, we obtain
    \begin{align*}
         K(f,\sigma(t);A_0,A_1)
         &\leq
         \left\|f_0^*\right\|_{\overline{A}_0}
        +
        \sigma(t)
        \left\|f_1^*\right\|_{\overline{A}_1}\\
        &\lesssim
        \left\|\chi_{(0,t)}f_0^*\right\|_{\overline{A}_0}
        +
        \left\|\chi_{(t,\infty)}f_0^*\right\|_{\overline{A}_0}
        +
        \sigma(t)\left(
        \left\|\chi_{(0,t)}f_1^*\right\|_{\overline{A}_1}
        +
        \left\|\chi_{(t,\infty)}f_1^*\right\|_{\overline{A}_1}
        \right)\\
        &=
        \left\|\chi_{(0,t)}(s)(f^*(s)-f^*(t))\right\|_{\overline{A}_0}
        +
        \sigma(t)\left(
        f^*(t)\left\|\chi_{(0,t)}\right\|_{\overline{A}_1}
        +
        \left\|\chi_{(t,\infty)}f^*\right\|_{\overline{A}_1}
        \right)\\
        &\leq
        \left\|\chi_{(0,t)}f^*\right\|_{\overline{A}_0}
        +
        \varphi_0(t)f^*(t)
        +\sigma(t)
        \left\|\chi_{(t,\infty)}f^*\right\|_{\overline{A}_1}\\
        &=
        \left\|\chi_{(0,t)}f^*\right\|_{\overline{A}_0}
        +
        \left\|\chi_{(0,t)}\right\|_{\overline{A}_0}f^*(t)
        +\sigma(t)
        \left\|\chi_{(t,\infty)}f^*\right\|_{\overline{A}_1}\\
        &\leq
        \left\|\chi_{(0,t)}f^*\right\|_{\overline{A}_0}
        +
        \left\|\chi_{(0,t)}(s)f^*(s)\right\|_{\overline{A}_0}
        +\sigma(t)
        \left\|\chi_{(t,\infty)}f^*\right\|_{\overline{A}_1}\\
        &\lesssim
        \left\|\chi_{(0,t)}f^*\right\|_{\overline{A}_0}
        +\sigma(t)
        \left\|\chi_{(t,\infty)}f^*\right\|_{\overline{A}_1}.
    \end{align*}
Combination of this with \eqref{E:RHS<LHS} establishes our claim.
\end{proof}

We shall now concentrate on proving Theorem~\ref{T:2}. We will proceed in several steps, formulating and proving auxiliary results of independent interest. First, we shall establish easily verifiable sufficient conditions for \eqref{E:GenProp1} and \eqref{E:GenProp2}. As we know thanks to Theorem~\ref{T:GeneralK}, validity of these conditions can be a decisive factor in the hunt for the formula describing the $K$-functional on the left-hand side of~\eqref{E:GeneralK}. 

\begin{lemma}
  \label{Lem:SufCondGeneralK}
 Let $A_0$ and $A_1$ be r.i.~quasi-Banach function spaces that admit, respectively, the representation spaces $\overline{A}_0$ and $\overline{A}_1$ over $(0,\infty)$.  For $i\in\{0,1\}$, let $\varphi_i(t)=\|\chi_{(0,t)}\|_{\overline{A}_i}$ for $t\in(0,\infty)$, and let  $ \sigma = \frac{\varphi_0}{\varphi_1}$.
 If \begin{equation}
        \label{E:SufCondGenProp1}
        \left\|\frac{\chi_{(0,t)}}{\varphi_1}\right\|_{\overline{A}_0}
        \lesssim 
        \sigma(t)
        \quad\text{for $t\in(0,\infty)$,}
    \end{equation} 
    then \eqref{E:GenProp1} holds.
    If \ \begin{equation}
        \label{E:SufCondGenProp2}
        \sigma(t)\left\|\frac{\chi_{(t,\infty)}}{\varphi_0}\right\|_{\overline{A}_1}
        \lesssim 
        1        
        \quad\text{for $t\in(0,\infty)$,}
    \end{equation} then \eqref{E:GenProp2} holds.
\end{lemma}

\begin{proof}
Let $t\in(0,\infty)$. First let us note that, owing to the boundedness of the dilation operator~\eqref{E:dilation}, one has, for any given function $f$, and any $s\in(0,\infty)$,
\begin{equation}
    \label{Eq1:AuxLemma1.3}
    f^*(\tfrac{s}{2})\varphi_1(s)=f^*(\tfrac{s}{2})\|\chi_{(0,s)}\|_{\overline{A}_1}\leq \|f^*(\tfrac{\tau}{2})\chi_{(0,s)}(\tau)\|_{\overline{A}_1}
\le \|f^*(\tfrac{\tau}{2})\|_{\overline{A}_1}
\lesssim
\left\|f^*\right\|_{\overline{A}_1}.
  \end{equation}
Thus, from \eqref{Eq1:AuxLemma1.3} and \eqref{E:SufCondGenProp1}, we obtain
\begin{align*}
        \left\|\chi_{(0,t)}(s)f^*(\tfrac{s}{2})\right\|_{\overline{A}_0}
        &
          =\left\|\chi_{(0,t)}(s)\frac{\varphi_1(s)}{\varphi_1(s)}f^*(\tfrac{s}{2})\right\|_{\overline{A}_0}
            \lesssim  
            \left\|f^*\right\|_{\overline{A}_1} 
            \left\|\frac{\chi_{(0,t)}}{\varphi_1}\right\|_{\overline{A}_0}
            \lesssim 
            \left\|f^*\right\|_{\overline{A}_1}\sigma(t),
    \end{align*}
which gives \eqref{E:GenProp1}.

Let us now prove the second assertion.
Fix $t\in(0,\infty)$. Similarly as in~\eqref{Eq1:AuxLemma1.3}, let us note that, for any  given function $f$, and any $s\in(0,\infty)$,
 \begin{equation}
 \label{Eq2:AuxLemma1.3}
 f^*(\tfrac{s}{2})\varphi_0(s)\leq \|f^*(\tfrac{\tau}{2})\chi_{(0,s)}(\tau)\|_{\overline{A}_0}
 \leq \|f^*(\tfrac{\tau}{2})\|_{\overline{A}_0}
 \lesssim
 \|f^*\|_{{\overline{A}}_0}.
  \end{equation}
  Using this and \eqref{E:SufCondGenProp2}, we obtain
\begin{align*}
       \sigma(t)\left\|\chi_{(t,\infty)}(s)f^*(\tfrac{s}{2})\right\|_{\overline{A}_1}
        &
          =\sigma(t)\left\|\chi_{(t,\infty)}(s)\frac{\varphi_0(s)}{\varphi_0(s)}f^*(\tfrac{s}{2})\right\|_{\overline{A}_1}
            \lesssim  
    \left\|f^*\right\|_{\overline{A}_0}\sigma(t)
    \left\|\frac{\chi_{(t,\infty)}}{\varphi_0}\right\|_{\overline{A}_1} 
    \lesssim 
   \left\|f^*\right\|_{\overline{A}_0},
    \end{align*}
    which gives \eqref{E:GenProp2}.
\end{proof}

In order to prove Theorem~\ref{T:2}, we need to focus on the specific situation when both $A_0$ and $A_1$ are spaces of type $\Lambda$. We shall show that, in such case, we can establish an explicit characterization of the $K$-functional from~\eqref{E:GeneralK}.  We begin with careful reformulation of the expression on the right-hand side of~\eqref{E:GeneralK} in an integral form. It is useful to notice that this step is not straightforward owing to the fact that the functional governing the space $\Lambda^{p_i}(w_i)$ requires rearranging the function in question. This affects significantly the second term on the right-hand side of~\eqref{E:GeneralK}.

In the following three lemmas, Convention~\ref{CON:representation} will be applied without further warning.

\begin{lemma}
    \label{lemma:equivalentKfunctIntegral}
    Assume that $p_0,p_1\in(0,\infty)$ and $w_0,w_1$ are weights on $(0,\infty)$.  
    For $i\in\{0,1\}$, let $\varphi_i(t)=\|\chi_{(0,t)}\|_{\Lambda^{p_i}(w_i)}$ for $t\in(0,\infty)$. 
    Let  $ \sigma = \frac{\varphi_0}{\varphi_1}$. Let $g\in \MRM$.

    {\rm (i) }
    One has
    \begin{align}
        &\left\|\chi_{(0,t)}g^*\right\|_{\Lambda^{p_0}(w_0)}
        +
        \sigma(t)\left\|\chi_{(t,\infty)}g^*\right\|_{\Lambda^{p_1}(w_1)}
        \label{E:lemma1.5-1}
        \\
        &\qquad\lesssim
        \left(\int_0^t g^*(s)^{p_0} w_0(s)\,ds\right)^{\frac{1}{p_0}} +  \sigma(t) \left(\int_t^{\infty} g^*(s)^{p_1} w_1(s)\,ds\right)^{\frac{1}{p_1}}\nonumber
    \end{align}
    with a constant in `$\lesssim$' independent of $g$ and $t\in(0,\infty)$.
            
    {\rm (ii) } If $\varphi_1$ satisfies the $\Delta_2$-condition, then 
        \begin{align}
       &\left(\int_0^t g^*(s)^{p_0} w_0(s)\,ds\right)^{\frac{1}{p_0}} +  \sigma(t) \left(\int_t^{\infty} g^*(s)^{p_1} w_1(s)\,ds\right)^{\frac{1}{p_1}}    \label{E:lemma1.5-2}
        \\
        &\qquad\lesssim
        \left\|\chi_{(0,t)}g^*\right\|_{\Lambda^{p_0}(w_0)}
        +
        \sigma(t)\left\|\chi_{(t,\infty)}g^*\right\|_{\Lambda^{p_1}(w_1)}.
        \nonumber
    \end{align}
         with a constant in `$\lesssim$' independent of $g$ and $t\in(0,\infty)$.
\end{lemma}

\begin{proof}
{\rm (i)} Let $g\in \MRM$. With no loss of generality we may assume that the expression on the right-hand side of~\eqref{E:lemma1.5-1} is finite. Fix $t\in (0, \infty)$ and define $g_t$ so that
\[ 
    g^*_t(s)=
             g^*(t)\chi_{(0,t]}(s)+  g^*(s) \chi_{(t,\infty)}(s)
             \quad\text{for $s\in(0,\infty)$.}
\]
            Due to the monotonicity of $g^*$, we have
             \begin{align*}
     &   \left\|\chi_{(0,t)}g^*\right\|_{\Lambda^{p_0}(w_0)}\!
        +
        \sigma(t)\left\|\chi_{(t,\infty)}g^*\right\|_{\Lambda^{p_1}(w_1)}  
        \le 
        \left(\int_0^t\! g^*(s)^{p_0} w_0(s)\,ds\right)^{\frac{1}{p_0}}\!\!\! +  \sigma(t) \left(\int_0^{\infty}\!\! g^*_t(s)^{p_1} w_1(s)\,ds\right)^{\frac{1}{p_1}}\\
        &\lesssim \left(\int_0^t\!\! g^*(s)^{p_0} w_0(s)\,ds\right)^{\frac{1}{p_0}}\!\!\! +  \sigma(t)\left(\int_0^{t}\!\! g^*(t)^{p_1} w_1(s)\,ds\right)^{\frac{1}{p_1}} \!\!\!+ \sigma(t)\left(\int_t^{\infty}\! \!g^*(s)^{p_1} w_1(s)\,ds\right)^{\frac{1}{p_1}}\\
           &=\left(\int_0^t g^*(s)^{p_0} w_0(s)\,ds\right)^{\frac{1}{p_0}} +  g^*(t) \varphi_0(t)+ \sigma(t)\left(\int_t^{\infty} g^*(s)^{p_1} w_1(s)\,ds\right)^{\frac{1}{p_1}}\\
           &\leq \left(\int_0^t g^*(s)^{p_0} w_0(s)\,ds\right)^{\frac{1}{p_0}} +  \left(\int_0^t g^*(s)^{p_0} w_0(s)\,ds\right)^{\frac{1}{p_0}}+ \sigma(t)\left(\int_t^{\infty} g^*(s)^{p_1} w_1(s)\,ds\right)^{\frac{1}{p_1}}\\
&\lesssim \left(\int_0^t g^*(s)^{p_0} w_0(s)\,ds\right)^{\frac{1}{p_0}} +  \sigma(t) \left(\int_t^{\infty} g^*(s)^{p_1} w_1(s)\,ds\right)^{\frac{1}{p_1}}
            \end{align*}
with the constant in $\lesssim$' depending only on $p_1$.

{\rm (ii)} Let $g\in \MRM$. Once again, we may assume that the expression on the right-hand side of~\eqref{E:lemma1.5-1} is finite. We fix $t\in(0,\infty)$ and define the function $g_t$ as above. Since $\varphi_1$ satisfies $\Delta_2$ - condition, $\Lambda^{p_1}(w_1)$ is quasinormed  space.  Therefore, using the monotonicity of  $g^*$,  we obtain 
\begin{align*}
    \left(\int_t^{\infty} g^*(s)^{p_1} w_1(s)\,ds\right)
    ^{\frac{1}{p_1}}
    &=
    \left(\int_t^{\infty} g_t^*(s)^{p_1} w_1(s)\,ds\right)
    ^{\frac{1}{p_1}} 
    \le
    \left(\int_0^{\infty} g_t^*(s)^{p_1} w_1(s)\,ds\right)
    ^{\frac{1}{p_1}} 
        \\
    &= \left\|g_t^*\right\|_{\Lambda^{p_1}(w_1)}
    \lesssim
    \left\|\chi_{(0,t)}g_t^*\right\|_{\Lambda^{p_1}(w_1)} 
    +
    \left\|\chi_{(t,\infty)}g_t^*\right\|_{\Lambda^{p_1}(w_1)}
        \\
    &=g^*(t)\left\|\chi_{(0,t)}\right\|_{\Lambda^{p_1}(w_1)} 
    +
    \left\|\chi_{(t,\infty)}g^*\right\|_{\Lambda^{p_1}(w_1)},
\end{align*}
and, in turn
\allowdisplaybreaks
\begin{align*}
     & \left(\int_0^t g^*(s)^{p_0} w_0(s)\,ds\right)^{\frac{1}{p_0}} +  \sigma(t) \left(\int_t^{\infty} g^*(s)^{p_1} w_1(s)\,ds\right)^{\frac{1}{p_1}}
        \\
    &\lesssim
    \left(\int_0^t g^*(s)^{p_0} w_0(s)\,ds\right)^{\frac{1}{p_0}} +  g^*(t) \varphi_0(t)+  \sigma(t)\left\|\chi_{(t,\infty)}g^*\right\|_{\Lambda^{p_1}(w_1)}
        \\
    &\leq \left(\int_0^t g^*(s)^{p_0} w_0(s)\,ds\right)^{\frac{1}{p_0}} +  \left(\int_0^t g^*(s)^{p_0} w_0(s)\,ds\right)^{\frac{1}{p_0}}+ \sigma(t)\left\|\chi_{(t,\infty)}g^*\right\|_{\Lambda^{p_1}(w_1)}
        \\
    &\lesssim 
    \left(\int_0^t g^*(s)^{p_0} w_0(s)\,ds\right)^{\frac{1}{p_0}} +   \sigma(t)\left\|\chi_{(t,\infty)}g^*\right\|_{\Lambda^{p_1}(w_1)}
        \\
    &=\left\|\chi_{(0,t)}g^*\right\|_{\Lambda^{p_0}(w_0)}
        +
    \sigma(t)\left\|\chi_{(t,\infty)}g^*\right\|_{\Lambda^{p_1}(w_1)},
\end{align*}
as desired.           
\end{proof}

Our next goal is to point out a rather remarkable connection between conditions~\eqref{E:SufCondGenProp1} and~\eqref{E:SufCondGenProp2} and a certain monotonicity property of the ratio involving the fundamental functions of spaces of type $\Lambda$. Let us note that in the case when 
\begin{equation*}
    A_0=\Lambda^{p_0}(w_0)
    \quad\text{and}\quad
    A_1=\Lambda^{p_1}(w_1),
\end{equation*}
then it follows straightforward from the definition of the functional $\|\cdot\|_{\Lambda^{p_i}(w_i)}$ and the monotonicity of the functions $\varphi_i$ that the condition~\eqref{E:SufCondGenProp1} reads as
\begin{equation}
    \label{E:first-for-lambda}
    \int_{0}^{t}
    \left(
    \frac{1}{\varphi_1(s)}\right)^{p_0}w_0(s)\,ds\lesssim \sigma(t)^{p_0}
    \quad\text{for $t\in(0,\infty)$,}
\end{equation}
and, likewise, the condition~\eqref{E:SufCondGenProp2}
reads as
\begin{equation}
    \label{E:second-for-lambda}
    \sigma(t)^{p_1}
    \int_{0}^{\infty}
    \left(
    \frac{1}{\varphi_0(s+t)}\right)
    ^{p_1}w_1(s)
    \,ds\lesssim 1
    \quad\text{for $t\in(0,\infty)$.}
\end{equation}
Now, while~\eqref{E:first-for-lambda}
is quite satisfactory for further applications,~\eqref{E:second-for-lambda} is not, owing to the shift in the argument of $\varphi_0$ caused by rearranging. For this reason, we will now establish an alternative characterization of~\eqref{E:second-for-lambda}, hence of \eqref{E:SufCondGenProp2} with $A_i=\Lambda^{p_i}(w_i)$ for $i\in\{0,1\}$ in this setting.

\begin{lemma}
  \label{Lem:SufCondGeneralKLambda}
 Assume that $p_0,p_1\in(0,\infty)$ and $w_0,w_1$ are weights on $(0,\infty)$. 
 For $i\in\{0,1\}$, let $\varphi_i(t)=\|\chi_{(0,t)}\|_{\Lambda^{p_i}(w_i)}$ for $t\in(0,\infty)$. Suppose that $\varphi_1$ satisfies the $\Delta_2$-condition. Let $ \sigma = \frac{\varphi_0}{\varphi_1}$.
Then, the condition~\eqref{E:second-for-lambda}, hence \eqref{E:SufCondGenProp2} with $A_i=\Lambda^{p_i}(w_i)$ for $i\in\{0,1\}$, 
    is equivalent to \begin{equation}
        \label{E:SufCondGenProp2**}
        \sigma(t)\left(\int_{t}^{\infty}\left(\frac{1}{\varphi_0(s)}\right)^{p_1}w_1(s)\,ds\right)^{\frac{1}{p_1}}
        \lesssim 
        1\quad\text{for $t\in(0,\infty)$,}
    \end{equation} 
    in which the constant in the relation $\lesssim$ does not depend on $t$.
\end{lemma}

\begin{proof}
Suppose first that \eqref{E:SufCondGenProp2} holds with $A_i=\Lambda^{p_i}(w_i)$ for $i\in\{0,1\}$. Since $\varphi_1$ satisfies the $\Delta_2$ condition, $\Lambda^{p_1}(w_1)$ is a quasinormed  space.  Therefore, using the monotonicity of  $\varphi_0$,  we have 
 
 \allowdisplaybreaks
    \begin{align*}
     \sigma(t)\left(\int_{t}^{\infty}\left(\frac{1}{\varphi_0(s)}\right)^{p_1}w_1(s)\,ds\right)^{\frac{1}{p_1}}
     &\le 
     \sigma(t)\left(\int_{0}^{\infty}\left(\min\left(\frac{1}{\varphi_0(t)}, \frac{1}{\varphi_0(s)}\right)\right)^{p_1}w_1(s)\,ds\right)^{\frac{1}{p_1}}\\
    & =\sigma(t) 
    \left\|\frac{\chi_{(0,t]}(s)}{\varphi_0(t)} + 
    \frac{\chi_{(t,\infty)}(s)}{\varphi_0(s)}\right\|_{\Lambda^{p_1}(w_1)}\\
          & \lesssim \sigma(t) \frac{1}{\varphi_0(t)} \left\|\chi_{(0,t]}\right\|_{\Lambda^{p_1}(w_1)}+  \sigma(t)\left\| \frac{\chi_{(t,\infty)}}{\varphi_0}\right\|_{\Lambda^{p_1}(w_1)}\\  
       &\lesssim  1,
       \end{align*}
        which gives \eqref{E:SufCondGenProp2**}.
 Conversely, suppose that \eqref{E:SufCondGenProp2**} holds. Then, by the monotonicity of $\varphi_0$ and the lattice property of $\left\| \cdot\right\|_{\Lambda^{p_1}(w_1)}$, we have
 \begin{align*}
    \sigma(t)
    & 
    \left\| \frac{\chi_{(t,\infty)}}{\varphi_0}\right\|_{\Lambda^{p_1}(w_1)}
    \le  
    \sigma(t)\left\| 
    \frac{\chi_{(0,t]}(s)}{\varphi_0(t)}+\frac{\chi_{(t,\infty)}(s)}{\varphi_0(s)}
    \right\|_{\Lambda^{p_1}(w_1)}
        \\
    &\lesssim  
    \sigma(t)\left(\int_{0}^{t}\left(\frac{1}{\varphi_0(t)}\right)^{p_1}w_1(s)\,ds\right)^{\frac{1}{p_1}} + \sigma(t)\left(\int_{t}^{\infty}\left(\frac{1}{\varphi_0(s)}\right)^{p_1}w_1(s)\,ds\right)^{\frac{1}{p_1}}\\
    &=  
    \sigma(t)\frac{1}{\varphi_0(t)}\left(\int_{0}^{t}w_1(s)\,ds\right)^{\frac{1}{p_1}} + \sigma(t)\left(\int_{t}^{\infty}\left(\frac{1}{\varphi_0(s)}\right)^{p_1}w_1(s)\,ds\right)^{\frac{1}{p_1}}\\
    &=  
    \sigma(t)\frac{\varphi_1(t)}{\varphi_0(t)} + \sigma(t)\left(\int_{t}^{\infty}\left(\frac{1}{\varphi_0(s)}\right)^{p_1}w_1(s)\,ds\right)^{\frac{1}{p_1}}\\
    &\lesssim  1,
\end{align*}
which gives \eqref{E:SufCondGenProp2} with $A_i=\Lambda^{p_i}(w_i)$ for $i\in\{0,1\}$. 
\end{proof}

\begin{lemma}
  \label{Lem:LorentLambdaSpaceSufCondGeneralK}
    Assume that $p_0,p_1\in(0,\infty)$ and $w_0,w_1$ are weights on $(0,\infty)$. 
    For $i\in\{0,1\}$, let $\varphi_i(t)=\|\chi_{(0,t)}\|_{\Lambda^{p_i}(w_i)}$ for $t\in(0,\infty)$. Suppose that each of the functions $\varphi_0,\varphi_1$ satisfies the $\Delta_2$-condition and that, for each $i\in\{0,1\}$, one has $\varphi_i(\infty)=\infty$.
    Let $ \sigma = \frac{\varphi_0}{\varphi_1}$.
    Then condition \eqref{E:SufCondGenProp1} with $A_i=\Lambda^{p_i}(w_i)$ for $i\in\{0,1\}$ is equivalent to
    \begin{equation} 
        \label
        {E:ratio-monotone}
        t
        \mapsto \frac{\sigma(t)}{\varphi_1(t)^{\varepsilon}}
        \quad\text{is equivalent to a non-decreasing function  on $(0,\infty)$ for some $\varepsilon>0$.}
    \end{equation}
    Moreover, \eqref{E:ratio-monotone} is also equivalent to  \eqref{E:SufCondGenProp2} with $A_i=\Lambda^{p_i}(w_i)$ for $i\in\{0,1\}$.
\end{lemma}

\begin{remark}
For particular tasks, it might be useful to notice that, under the assumptions of Lemma~\ref{Lem:LorentLambdaSpaceSufCondGeneralK}, if
\begin{equation} 
    \label{E:fundamental-mon}
        t
        \mapsto \frac{\varphi_1(t)}{t}
        \quad\text{is equivalent to a non-increasing function on $(0,\infty)$,}
    \end{equation}
then the condition
\begin{equation} 
    \label{E:mon}
        t
        \mapsto \frac{\sigma(t)}{t^{\varepsilon}}
        \quad\text{is equivalent to a non-decreasing function  on $(0,\infty)$ for some $\varepsilon>0$}
    \end{equation}
    is sufficient for
    \eqref{E:ratio-monotone}.
    As is well known (\cite[Chapter~2, Corollary~5.3, page~67]{Ben:88}),~\eqref{E:fundamental-mon} is automatically satisfied whenever $\Lambda^{p_1}(w_1)$ is equivalent to a normed space. Consequently, in all such cases, it is sufficient to verify~\eqref{E:mon} rather than~\eqref{E:fundamental-mon}. Characterization of those $p_1$ and $w_1$ for which $\Lambda^{p_1}(w_1)$ is equivalent to a normed space can be found in~\cite[Theorem 4]{Saw:90}.
\end{remark}

\begin{proof}[Proof of Lemma~\ref{Lem:LorentLambdaSpaceSufCondGeneralK}]
First suppose that \eqref{E:ratio-monotone} holds with some $\varepsilon>0$.
Since
$$
\varphi_i(t)=\left(\int_0^t w_i(s)\,ds\right)^{\frac{1}{p_i}} \quad\text{for $i\in\{0,1\}$ and $t\in(0,\infty)$,}
$$
one has 
\begin{equation}
    \label{E:derivative}
    \frac{d\varphi_i^{p_i}}{dt}=w_i
    \quad\text{for $i\in\{0,1\}$ on $(0,\infty)$.}
\end{equation}
Fix $t\in(0,\infty)$, Then, by~\eqref{E:ratio-monotone},
\begin{equation}
\label{E:lemma-3.5-1}
    \left\|\frac{\chi_{(0,t)}}{\varphi_1}\right\|_{\Lambda^{p_0}(w_0)}
        =
    \left\|\chi_{(0,t)}
    \frac{\varphi_0^{\tfrac{1}{1+\varepsilon}}}{\varphi_1
    }\varphi_0^{-\tfrac{1}{1+\varepsilon}}\right\|_{\Lambda^{p_0}(w_0)}
        \lesssim 
   \frac{\varphi_0(t)^{\tfrac{1}{1+\varepsilon}}}{\varphi_1(t)}  
   \left\|\chi_{(0,t)}\varphi_0^{-\tfrac{1}{1+\varepsilon}}\right\|_{\Lambda^{p_0}(w_0)}.
\end{equation}
Using~\eqref{E:derivative}, we obtain
\begin{align*}
    \left\|\chi_{(0,t)}\varphi_0^{-\tfrac{1}{1+\varepsilon}}\right\|_{\Lambda^{p_0}(w_0)}
        =
    \left(\int_0^t 
    \left(\varphi_0(s)^{p_0}
    \right)^{-\tfrac{1}{1+\varepsilon}}\, w_0(s)\,ds\right)^{\frac{1}{p_0}}
        \approx
   \varphi_0(t)^{-\tfrac{1}{1+\varepsilon}+1}.
\end{align*}
Plugging this into~\eqref{E:lemma-3.5-1} and using the definition of $\sigma$, we arrive at
\begin{align*}
    \left\|\frac{\chi_{(0,t)}}{\varphi_1}\right\|_{\Lambda^{p_0}(w_0)}
    &\lesssim
 \sigma(t)
 \quad\text{for $t\in(0,\infty)$}
\end{align*}
with a constant in `$\lesssim$' independent of $t$. Therefore, \eqref{E:first-for-lambda}, or, which is the same, \eqref{E:SufCondGenProp1} with $A_i=\Lambda^{p_i}(w_i)$ for $i\in\{0,1\}$, follows.

Now suppose that \eqref{E:SufCondGenProp1} with $A_i=\Lambda^{p_i}(w_i)$ for $i\in\{0,1\}$ is true. As we know, that means that~\eqref{E:first-for-lambda} holds. Thus, there is a $C>0$ such that
\begin{equation}
        \label{E:SufCondGenProp1AuxInt}
\frac{1}{\varphi_0(t)^{p_0}}\int_0^t\frac{1}{\varphi_1(s)^{p_0}}d\varphi_0^{p_0}(s) \leq C\frac{1}{\varphi_1(t)^{p_0}}
\quad \text{for $\in(0,\infty)$.}
    \end{equation}
From this, we obtain by integration that
\begin{equation*}
\int_0^x\frac{1}{\varphi_0(t)^{p_0}}\int_0^t\frac{1}{\varphi_1(s)^{p_0}}d\varphi_0^{p_0}(s) d\varphi_0^{p_0}(t) \leq C \int_0^x\frac{1}{\varphi_1(t)^{p_0}} d\varphi_0^{p_0}(t)\quad 
\text{for $x\in(0,\infty)$.} 
\end{equation*}
Using Fubini's theorem on the left term of the preceding inequality and applying~\eqref{E:SufCondGenProp1AuxInt} to the right one, we obtain
\begin{equation*}
        \label{E:SufCondGenProp1AuxInt2}
\int_0^x\frac{1}{\varphi_1(s)^{p_0}}\int_s^x\frac{1}{\varphi_0(t)^{p_0}}d\varphi_0^{p_0}(t) d\varphi_0^{p_0}(s) \leq C^2\; \frac{\varphi_0(x)^{p_0}}{\varphi_1(x)^{p_0}}
\quad\text{for $x\in(0,\infty)$,} 
    \end{equation*}
    which gives
\begin{equation*}
\frac{1}{\varphi_0(x)^{p_0}}\int_0^x \log\Big(\frac{\varphi_0(x)^{p_0}}{\varphi_0(s)^{p_0}}\Big) \frac{1}{\varphi_1(s)^{p_0}} d\varphi_0^{p_0}(s) \leq C^2\; \frac{1}{\varphi_1(x)^{p_0}} \quad\text{for $x\in(0,\infty)$.} 
    \end{equation*}
Again,  from this, we obtain by integration that   
 \begin{equation*}
\int_0^t\frac{1}{\varphi_0(x)^{p_0}}\int_0^x \log\Big(\frac{\varphi_0(x)^{p_0}}{\varphi_0(s)^{p_0}}\Big) \frac{1}{\varphi_1(s)^{p_0}} d\varphi_0^{p_0}(s) d\varphi_0^{p_0}(x) 
\leq C^2 
\int_0^t\frac{1}{\varphi_1(x)^{p_0}} d\varphi_0^{p_0}(x)
\quad\text{for $t\in(0,\infty)$.} 
    \end{equation*}
Using again Fubini's theorem  and~\eqref{E:SufCondGenProp1AuxInt}, we get
 \begin{equation*}
\int_0^t\frac{1}{\varphi_1(s)^{p_0}}
\int_s^t \log\Big(\frac{\varphi_0(x)^{p_0}}{\varphi_0(s)^{p_0}}\Big) \frac{1}{\varphi_0(x)^{p_0}} 
d\varphi_0^{p_0}(x) d\varphi_0^{p_0}(s) \leq C^3\; \frac{\varphi_0(t)^{p_0}}{\varphi_1(t)^{p_0}}
\quad\text{for $t\in(0,\infty)$,} 
\end{equation*} 
which gives
\begin{equation*}
\frac{1}{\varphi_0(t)^{p_0}}\int_0^t \frac{1}{2}\log^2\Big(\frac{\varphi_0(t)^{p_0}}{\varphi_0(s)^{p_0}}\Big) \frac{1}{\varphi_1(s)^{p_0}} 
d\varphi_0^{p_0}(s) \leq C^3\; \frac{1}{\varphi_1(t)^{p_0}}
\quad\text{for $t\in(0,\infty)$.} 
    \end{equation*}  
    Iterating this process $k$ times for some $k\in\N$, we derive
   \begin{equation}
 \label{E:SufCondGenProp1AuxInt5}
\int_0^t \frac{1}{k!}\log^k\Big(\frac{\varphi_0(t)^{p_0}}{\varphi_0(s)^{p_0}}\Big) \frac{1}{\varphi_1(s)^{p_0}} d\varphi_0^{p_0}(s) 
\leq C^{k+1}\; \frac{\varphi_0(t)^{p_0}}{\varphi_1(t)^{p_0}}
\quad\text{for $t\in(0,\infty)$.} 
    \end{equation}
    We choose $\varepsilon_0 \in(0,1)$ such that $\varepsilon_0 C<1$. Then, from \eqref{E:SufCondGenProp1AuxInt5}, we infer that
\begin{equation*}
 \label{E:SufCondGenProp1AuxInt6}
\sum_{k=0}^{\infty}
\int_0^t \frac{\varepsilon_0^k}{k!} 
\log^k\Big(\frac{\varphi_0(t)^{p_0}}{\varphi_0(s)^{p_0}}\Big) 
\frac{1}{\varphi_1(s)^{p_0}} 
d\varphi_0^{p_0}(s) 
\leq \left(\sum_{k=0}^{\infty} C^{k+1} \varepsilon_0^k\right)  \; \frac{\varphi_0(t)^{p_0}}{\varphi_1(t)^{p_0}}
\quad\text{for $t\in(0,\infty)$,} 
    \end{equation*}
          which gives, after interchanging summation  with integral, for some $C_1>0$,
    \begin{equation*}
\int_0^t \Big(\frac{\varphi_0(t)^{p_0}}{\varphi_0(s)^{p_0}}\Big)^{\varepsilon_0} 
\frac{1}{\varphi_1(s)^{p_0}} 
d\varphi_0^{p_0}(s) 
\leq C_1 \; 
\frac{\varphi_0(t)^{p_0}}{\varphi_1(t)^{p_0}}
\quad\text{for $t\in(0,\infty)$,} 
    \end{equation*}
or
    \begin{equation}
 \label{E:SufCondGenProp1AuxInt7}
\int_0^t \Big(\frac{1}{\varphi_0(s)^{p_0}}\Big)^{\varepsilon_0} \frac{1}{\varphi_1(s)^{p_0}} 
d\varphi_0^{p_0}(s) \leq C_1 \; 
\frac{\varphi_0(t)^{p_0(1-\varepsilon_0)}}{\varphi_1(t)^{p_0}}
\quad\text{for $t\in(0,\infty)$.} 
    \end{equation}
By monotonicity, for any $t\in(0,\infty)$, we have that
$$\int_0^t \Big(\frac{1}{\varphi_0(s)^{p_0}}\Big)^{\varepsilon_0} \frac{1}{\varphi_1(s)^{p_0}} 
d\varphi_0^{p_0}(s) 
\geq 
\frac{1}{\varphi_1(t)^{p_0}} 
\int_0^t \Big(\frac{1}{\varphi_0(s)^{p_0}}\Big)^{\varepsilon_0} d\varphi_0^{p_0}(s) \approx 
\frac{\varphi_0(t)^{p_0(1-\varepsilon_0)}}{\varphi_1(t)^{p_0}}.
$$  
From this and from \eqref{E:SufCondGenProp1AuxInt7}, 
we conclude that 
\begin{equation*}
\int_0^t 
\Big(\frac{1}{\varphi_0(s)^{p_0}}\Big)^{\varepsilon_0} 
\frac{1}{\varphi_1(s)^{p_0}} 
d\varphi_0^{p_0}(s)\approx \; 
\frac{\varphi_0(t)^{p_0(1-\varepsilon_0)}}{\varphi_1(t)^{p_0}}
\quad\text{for $t\in(0,\infty)$.} 
    \end{equation*}
Hence, 
 \begin{equation} 
        t
        \mapsto \frac{\varphi_0(t)^{p_0(1-\varepsilon_0)}}{\varphi_1(t)^{p_0}}
        \quad\text{is equivalent to a non-decreasing function on $(0,\infty)$,}
    \end{equation}
    which implies \eqref{E:ratio-monotone} with $\varepsilon >0$ given by $1+\varepsilon=1/(1-\varepsilon_0)$.

We shall now show 
that~ \eqref{E:ratio-monotone}
implies~\eqref{E:SufCondGenProp2} with $A_i=\Lambda^{p_i}(w_i)$ for $i\in\{0,1\}$.
Note that if \eqref{E:ratio-monotone} holds for some $\varepsilon>0$, then 
$$
        t
        \mapsto \frac{\sigma(t)}{\varphi_1(t)^{\varepsilon_1}}
        \quad\text{is equivalent to a non-decreasing function on $(0,\infty)$}
$$
for any $\varepsilon_1\in(0,\varepsilon)$. By Lemma \ref{Lem:SufCondGeneralKLambda}, \eqref{E:SufCondGenProp2} with $A_i=\Lambda^{p_i}(w_i)$ for $i\in\{0,1\}$ is equivalent to \eqref{E:SufCondGenProp2**}. So, it suffices prove \eqref{E:SufCondGenProp2**}.
Fix $t\in(0,\infty)$. Then,

\begin{align*}
 \sigma(t)\left(\int_{t}^{\infty}\left(\frac{1}{\varphi_0(s)}\right)^{p_1}w_1(s)\,ds\right)^{\frac{1}{p_1}}
 &=  
  \sigma(t)\left(\int_{t}^{\infty}\left(\frac{\varphi_1(s)^{1+\varepsilon}}{\varphi_0(s)}\frac{1}{\varphi_1(s)^{1+\varepsilon}}\right)^{p_1}w_1(s)\,ds\right)^{\frac{1}{p_1}}\\
      &\lesssim  
      \sigma(t) \frac{\varphi_1(t)^{1+\varepsilon}}{\varphi_0(t)}   \left(\int_t^\infty (\varphi_1(s)^{p_1})^{-(1+\varepsilon)}\, w_1(s)\,ds\right)^{\frac{1}{p_1}}\\
    &\approx  
    \sigma(t) \frac{\varphi_1(t)^{1+\varepsilon}}{\varphi_0(t)}   \varphi_1(t))^{-\varepsilon} \\
     &= 1,
\end{align*}
which gives \eqref{E:SufCondGenProp2**}. Hence, \eqref{E:SufCondGenProp2} with $A_i=\Lambda^{p_i}(w_i)$ for $i\in\{0,1\}$ follows as well.

It remains to show that \eqref{E:SufCondGenProp2} with $A_i=\Lambda^{p_i}(w_i)$ for $i\in\{0,1\}$ implies~\eqref{E:ratio-monotone}. To this end, suppose  that \eqref{E:SufCondGenProp2} with $A_i=\Lambda^{p_i}(w_i)$ for $i\in\{0,1\}$ is true. 
Then, we also have \eqref{E:SufCondGenProp2**}.
Therefore, there is $C>0$, such that
    \begin{equation}
    \label{E:SufCondGenProp2AuxInt}
\frac{1}{\varphi_1(t)^{p_1}}  \int_t^{\infty}\frac{1}{\varphi_0(s)^{p_1} }d\varphi^{p_1}_1(s) \leq C\frac{1}{\varphi_0(t)^{p_1} }
\quad\text{for $t\in(0,\infty)$.} 
    \end{equation}
From this, we obtain by integration that
\begin{equation*}
\int_x^{\infty}
\frac{1}{\varphi_1(t)^{p_1}}  
\int_t^{\infty}
\frac{1}{\varphi_0(s)^{p_1}}
d\varphi_1^{p_1}(s)  d\varphi_1^{p_1}(t)
\leq C 
\int_x^{\infty}\frac{1}{\varphi_0(t)^{p_1}}
d\varphi_1^{p_1}(t)
\quad\text{for $x\in(0,\infty)$.} 
    \end{equation*}
Applying Fubini's theorem to the left term of the previous inequality, and the estimate \eqref{E:SufCondGenProp2AuxInt} to the right one, we obtain
\begin{equation*}
        \label{E:SufCondGenProp2AuxInt2}
\int_x^{\infty}\frac{1}{\varphi_0(s)^{p_1}}
\int_x^s\frac{1}{\varphi_1(t)^{p_1}}
d\varphi_1^{p_1}(t) d\varphi_1^{p_1}(s) 
\leq C^2\; 
\frac{\varphi_1(x)^{p_1}}{\varphi_0(x)^{p_1}}
\quad\text{for $x\in(0,\infty)$,} 
    \end{equation*}
    which gives
\begin{equation*}
\frac{1}{\varphi_1(x)^{p_1}}
\int_x^{\infty} 
\log\Big(\frac{\varphi_1(s)^{p_1}}{\varphi_1(x)^{p_1}}\Big) 
\frac{1}{\varphi_0(s)^{p_1}} 
d\varphi_1^{p_1}(s) 
\leq C^2\; 
\frac{1}{\varphi_0(x)^{p_1}}
\quad\text{for $x\in(0,\infty)$.} 
    \end{equation*}
Again, from this, we obtain by integration that   
 \begin{equation*}
\int_t^{\infty}
\frac{1}{\varphi_1(x)^{p_1}}
\int_x^{\infty} 
\log\Big(\frac{\varphi_1(s)^{p_1}}
{\varphi_1(x)^{p_1}}\Big) \frac{1}{\varphi_0(s)^{p_1}} d\varphi_1^{p_1}(s) d\varphi_1^{p_1}(x) 
\leq C^2 
\int_t^{\infty}\frac{1}{\varphi_0(x)^{p_1}} 
d\varphi_1^{p_1}(x)
\quad\text{for $t\in(0,\infty)$.} 
    \end{equation*}
Repeating the argument, we obtain
 \begin{equation*}
\int_t^{\infty}
\frac{1}{\varphi_0(s)^{p_1}}
\int_t^s \log\Big(\frac{\varphi_1(s)^{p_1}}{\varphi_1(x)^{p_1}}\Big) \frac{1}{\varphi_1(x)^{p_1}} 
d\varphi_1^{p_1}(x) d\varphi_1^{p_1}(s) 
\leq C^3\; 
\frac{\varphi_1(t)^{p_1}}{\varphi_0(t)^{p_1}}
\quad\text{for $t\in(0,\infty)$,} 
    \end{equation*} 
      which gives
\begin{equation*}
\frac{1}{\varphi_1(t)^{p_1}}
\int_t^{\infty} 
\frac{1}{2}\log^2\Big(\frac{\varphi_1(s)^{p_1}}
{\varphi_1(t)^{p_1}}\Big) \frac{1}{\varphi_1(s)^{p_1}}
d\varphi_1^{p_1}(s) 
\leq C^3\; 
\frac{1}{\varphi_0(t)^{p_1}}
\quad\text{for $t\in(0,\infty)$.} 
    \end{equation*}  
    Iterating the argument $k$ times, where $k\in\N$, one derives
   \begin{equation}
 \label{E:SufCondGenProp2AuxInt5}
\int_t^{\infty} 
\frac{1}{k!}
\log^k\Big(\frac{\varphi_1(s)^{p_1}}
{\varphi_1(t)^{p_1}}\Big) 
\frac{1}{\varphi_0(s)^{p_1}} 
d\varphi_1^{p_1}(s) 
\leq C^{k+1}\; 
\frac{\varphi_1(t)^{p_1}}
{\varphi_0(t)^{p_1}}
\quad\text{for $t\in(0,\infty)$.} 
    \end{equation}
    Fix $\varepsilon>0$ such that $\varepsilon C<1$. Then, from \eqref{E:SufCondGenProp2AuxInt5},  one obtains that
\begin{equation*}
 \label{E:SufCondGenProp2AuxInt6}
\sum_{k=0}^{\infty}\int_t^{\infty} 
\frac{ \varepsilon^k}{k!}
\log^k\Big(\frac{\varphi_1(s)^{p_1}}
{\varphi_1(t)^{p_1}}\Big) 
\frac{1}{\varphi_0(s)^{p_1}} 
d\varphi_1^{p_1}(s) 
\leq 
\left(\sum_{k=0}^{\infty} C^{k+1} \varepsilon^k\right)  \; \frac{\varphi_0(t)^{p_0}}{\varphi_1(t)^{p_0}}
\quad\text{for $t\in(0,\infty)$,} 
    \end{equation*}
    which now gives, after an interchange of summation  with integral, for some $C_1>0$,
    \begin{equation*}
\int_t^{\infty} 
\Big(\frac{\varphi_1(s)^{p_1}}
{\varphi_1(t)^{p_1}}\Big)^{\varepsilon} 
\frac{1}{\varphi_0(s)^{p_1}} 
d\varphi_1^{p_1}(s) 
\leq C_1 \; 
\frac{\varphi_1(t)^{p_1}}{\varphi_0(t)^{p_1}}
\quad\text{for $t\in(0,\infty)$,} 
    \end{equation*}
which it is equivalent to 
    \begin{equation}
 \label{E:SufCondGenProp2AuxInt7}
\int_t^{\infty}  
\Big(\varphi_1(s)^{p_1}\Big)^{\varepsilon} 
\frac{1}{\varphi_0(s)^{p_1}} 
d\varphi_1^{p_1}(s) 
\leq C_1 \; 
\frac{\varphi_1(t)^{p_1(1+\varepsilon)}}{\varphi_0(t)^{p_1}}
\quad\text{for $t\in(0,\infty)$.} 
    \end{equation}
   From this it also follows 
   \begin{equation*}  
 \frac{1}{\varphi_0(t)^{p_1}}
 \int_0^{t}  
 \Big(\varphi_1(s)^{p_1}\Big)^{\varepsilon}
 d\varphi_1^{p_1}(s)
 + 
 \int_t^{\infty}  
 \Big(\varphi_1(s)^{p_1}\Big)^{\varepsilon} 
 \frac{1}{\varphi_0(s)^{p_1}} 
 d\varphi_1^{p_1}(s) 
 \leq C_2 \; 
 \frac{\varphi_1(t)^{p_1(1+\varepsilon)}}{\varphi_0(t)^{p_1}}
 \quad\text{for $t\in(0,\infty)$.} 
    \end{equation*}
    Owing to the fact that the reverse estimate is  trivial, we in fact have 
      \begin{equation}  \label{E:SufCondGenProp2AuxInt9}
 \frac{1}{\varphi_0(t)^{p_1}}
 \int_0^{t}  
 \Big(\varphi_1(s)^{p_1}\Big)^{\varepsilon}
 d\varphi_1^{p_1}(s)
 + 
 \int_t^{\infty}  
 \Big(\varphi_1(s)^{p_1}\Big)^{\varepsilon} 
 \frac{1}{\varphi_0(s)^{p_1}} 
 d\varphi_1^{p_1}(s) 
 \approx  
 \frac{\varphi_1(t)^{p_1(1+\varepsilon)}}{\varphi_0(t)^{p_1}}
 \quad\text{for $t\in(0,\infty)$.} 
    \end{equation}
    Since the left hand side of \eqref{E:SufCondGenProp2AuxInt9} is equal to
    \begin{equation*}
         \int_0^{\infty}  \Big(\varphi_1(s)^{p_1} \Big)^{\varepsilon} \min\Big\{\frac{1}{\varphi_0(t)^{p_1}},\frac{1}{\varphi_0 (s)^{p_1}}\Big\} d\varphi_1^{p_1}(s)\quad\text{for $t\in(0,\infty)$,} 
    \end{equation*}
 we conclude that
 \begin{equation} 
    \mapsto\frac{\varphi_0(t)^{p_1}}{\varphi_1(t)
        ^{p_1(1+\varepsilon)}}
        \quad\text{is equivalent to a non-decreasing function on $(0,\infty)$,}
    \end{equation}
    which implies \eqref{E:ratio-monotone}. 
\end{proof}

\begin{proof}[Proof of Theorem~\ref{T:2}]
Fix $f\in S^{p_0}(w_0)+S^{p_1}(w_1)$. Then, 
    by Theorem~\ref{T:1.1} (note that every $f$ in $S^{p_0}(w_0)+S^{p_1}(w_1)$ satisfies $f^*\in A$), one has
    \begin{align}
        \label{E:doublestar}
        &K(f, \theta(t);S^{p_0}(w_0),S^{p_1}(w_1))
        \approx
        K(Tf^*, \theta(t);\Lambda^{p_0}(\widetilde{w_0}),\Lambda^{p_1}(\widetilde{w_1}))
        \quad\text{for $t\in(0,\infty)$.}
    \end{align}
    For each $i\in\{0,1\}$, let $\varphi_i$ be the fundamental function of $\Lambda^{p_i}(\widetilde{w_i})$ and let $\widetilde{\sigma}=\frac{\varphi_1}{\varphi_0}$. Then the change of variables $s\mapsto \frac1s$ in the integral in~\eqref{E:varphi} gives
    \begin{equation*}
        \varphi_i(t)=\left(
        \int_{0}^{t}
        \widetilde{w_i}(s)\dd s
        \right)^{\ipi}
        =
        \psi_i(\tfrac{1}{t})
        \quad\text{for $i\in\{0,1\}$ and every $t\in(0,\infty)$,}
    \end{equation*}
    and, in turn,
    \begin{equation}
        \label{E:omega-sigma}
        \theta (t) = \frac{1}{\tilde\sigma\left(\frac{1}{t}\right)}
        \quad\text{for $t\in(0,\infty)$.}
    \end{equation}
    Consequently, the function
    \begin{equation*}
        t\mapsto \frac{\widetilde{\sigma}(t)}{\varphi_0(t)^{\varepsilon}}=
        \frac{\left(
        \int_{0}^{t}
        \widetilde{w_1}(s)\dd s
        \right)^{\ipon}}
        {\left(
        \int_{0}^{t}
        \widetilde{w_0}(s)\dd s
        \right)^{\ipo(1+\varepsilon)}}
        =
        \frac{1}{\theta(\frac{1}{t})\psi_0(\frac{1}{t})^{\varepsilon}}
    \end{equation*}
    is, by the assumption \eqref{E:cond3}, equivalent to a non-decreasing function on $(0,\infty)$.
    Therefore, the condition~\eqref{E:ratio-monotone} is satisfied with $\sigma$ and $\varphi_1$ replaced by $\tilde\sigma$ and $\varphi_0$, respectively. This implies, via Lemma~\ref{Lem:LorentLambdaSpaceSufCondGeneralK}, that both the conditions~\eqref{E:SufCondGenProp1} and~\eqref{E:SufCondGenProp2} hold, with $\sigma$ and $\varphi_1$ replaced by $\tilde\sigma$ and $\varphi_0$, respectively. This means, owing to Lemma~\ref{Lem:SufCondGeneralK}, that both the conditions~\eqref{E:GenProp1} and~\eqref{E:GenProp2} hold with $\sigma$ replaced by $\tilde\sigma$ and with 
    $A_0=\Lambda^{p_1}(\widetilde{w_1})$, 
    $A_1=\Lambda^{p_0}(\widetilde{w_0})$. Therefore, by Theorem~\ref{T:GeneralK}, applied to the couple $(\Lambda^{p_1}(\widetilde{w_1}),\Lambda^{p_0}(\widetilde{w_0}))$ in place of 
    $(A_0,A_1)$, we obtain
        \begin{align*}
        &K\left(g,\widetilde{\sigma}(t);\Lambda^{p_1}(\widetilde{w_1}),\Lambda^{p_0}(\widetilde{w_0})\right)
            \approx
        \left\|\chi_{(0,t)}g^*\right\|_{\Lambda^{p_1}(\widetilde{w_1})}
        +
        \widetilde\sigma(t)
        \left\|\chi_{(t,\infty)}g^*\right\|_{\Lambda^{p_0}(\widetilde{w_0})}
    \end{align*}
    for every $g\in \Lambda^{p_1}(\widetilde{w_1})
    +\Lambda^{p_0}(\widetilde{w_0})$ and every $t\in(0,\infty)$.
    Thus, by Lemma~\ref{lemma:equivalentKfunctIntegral}, 
   \begin{align}
        \label{E:K-swapped}
        &K\left(g,\widetilde{\sigma}(t);\Lambda^{p_1}(\widetilde{w_1}),\Lambda^{p_0}(\widetilde{w_0})\right)
           \approx
        \left(
        \int_{0}^{t}
        g^*(s)^{p_1}
        \widetilde{w_1}(s)\dd s
        \right)^{\ipon}
            +
        \widetilde{\sigma}(t)
        \left(
        \int_{t}^{\infty}
        g^*(s)^{p_0}
        \widetilde{w_0}(s)\dd s
        \right)^{\ipo}
    \end{align}
    for every $g\in \Lambda^{p_1}(\widetilde{w_1})
    +\Lambda^{p_0}(\widetilde{w_0})$ and every $t\in(0,\infty)$. Since
    \begin{equation*}
        K(f,s;X_0,X_1)
        =
        s\,K(f,\tfrac{1}{s};X_1,X_0)
        \quad\text{for any admissible $f$ and $s\in(0,\infty)$,}
    \end{equation*}
    using~\eqref{E:omega-sigma},~\eqref{E:K-swapped},~\eqref{E:prequel-S-Lambda} and \eqref{E:prequel-S-Lambda-dual}, we finally arrive at
    \begin{align}
        \label{E:K-final}
        K&\left(Tf^*,\theta(t);\Lambda^{p_0}(\widetilde{w_0}),\Lambda^{p_1}(\widetilde{w_1})\right)
        =
        \theta(t)
        K\left(Tf^*,\frac{1}{\theta(t)};
        \Lambda^{p_1}(\widetilde{w_1}),
        \Lambda^{p_0}(\widetilde{w_0})\right)
            \\
            &=
        \theta(t)
        K\left(Tf^*,\widetilde{\sigma}(\tfrac{1}{t});
        \Lambda^{p_1}(\widetilde{w_1}),
        \Lambda^{p_0}(\widetilde{w_0})\right) \nonumber
            \\
        &\approx
        \theta(t)
        \left(
        \int_{0}^{\frac{1}{t}}
        Tf^*(s)^{p_1}
        \widetilde{w_1}(s)\dd s
        \right)^{\ipon}
        +\theta(t)\widetilde{\sigma}(\tfrac{1}{t})
        \left(
        \int_{\frac{1}{t}}^{\infty}
        Tf^*(s)^{p_0}
        \widetilde{w_0}(s)\dd s
        \right)^{\ipo} \nonumber
            \\
        &=
        \theta(t)
        \left(
        \int_{t}^{\infty}
        \left(
        f^{**}(s)-f^*(s)\right)^{p_1}
        w_1(s)\dd s
        \right)^{\ipon}
        +
        \left(
        \int_{0}^{t}
        \left(
        f^{**}(s)-f^*(s)\right)^{p_0}
        w_0(s)\dd s
        \right)^{\ipo}  \nonumber
    \end{align}
    for $f\in S^{p_0}(\widetilde{w_0})
    +S^{p_1}(\widetilde{w_1})$ and $t\in(0,\infty)$.
    The assertion now follows from~\eqref{E:K-final} and~\eqref{E:doublestar}.
\end{proof}

\begin{proof}[Proof of Corollary~\ref{C:1}]
    We will show that the assertion follows as a special case of~\eqref{E:K-for-S}, in which we take $p_0=p_1=p$, $w_0(t)=1$ and $w_1(t)=t^{-\alpha}$ for $t\in(0,\infty)$. We only have to verify that, for this choice of parameters, the assumptions of Theorem~\ref{T:2} are satisfied. 
    
    First, observe that $S^{p_0}(w_0)=L^p$ by~\cite[Chapter~5, Proposition~7.12, page~384]{Ben:88}. Next, one has
    \begin{equation}
        \label{E:psi-equivalent-to-powers}
        \psi_0(t)\approx t^{\frac{1-p}{p}},\qquad \psi_1(t)\approx t^{\frac{1-p-\alpha}{p}}
        \quad\text{for $t\in(0,\infty)$.}
    \end{equation}
    Consequently, since $p>1$, the assumption~\eqref{E:psi-infinity-at-zero} is satisfied. 
    The condition~\eqref{E:cond1} is obviously valid due to the fact that both $\psi_0$ and $\psi_1$ are equivalent to decreasing power functions. The condition~\eqref{E:cond2} amounts to the existence of a positive constant $C$ such that
    \begin{equation*}
        \int_{0}^{t}1\,ds
        \le C
        t^{p}\int_{t}^{\infty}s^{-p}\,ds
        \quad\text{for $t\in(0,\infty)$,}
    \end{equation*}
    which is evident, once again thanks to the fact that $p>1$.
    Finally, owing to~\eqref{E:psi-equivalent-to-powers}, the function $\theta$ from~\eqref{E:sigma} satisfies
    $\theta(t)\approx t^{\frac{\alpha}{p}}$ for $t\in(0,\infty)$ with constants of equivalence independent of $t$. Therefore, the condition~\eqref{E:cond3} is satisfied with any choice of $\varepsilon\in(0,\frac{\alpha}{p-1})$.
\end{proof}

\begin{proof}[Proof of Corollary~\ref{C:2}]
    The proof is based on a repeated alteration of~\eqref{E:kolyada-inequality-2} and~\eqref{E:formula-rearrangement}.
    Let $k\in\N$, $k\ge 2$, and let $f$ be a $(k-1)$-times weakly differentiable function which satisfies~\eqref{E:condition-at-infty}. We claim that 
    \begin{equation}
        \label{E:kolyada-BS}
        f^{**}(t)-f^*(t)
        \le
        \frac{n^{\frac32k-2}}{(k-2)!}
        t^{\frac{1}{n}}
        \int_{t}^{\infty}
        s^{\frac{k-1}{n}-1}
        |\nabla^k f|^{**}(s)\dd s
        \quad\text{for $t\in(0,\infty)$.}
    \end{equation}
    We shall prove~\eqref{E:kolyada-BS} by induction. Fix $t\in(0,\infty)$. By~\eqref{E:kolyada-inequality-2} (used twice) and~\eqref{E:formula-rearrangement},    
    \begin{align*}
        f^{**}(t)-f^*(t)
        &\le
        \sqrt{n} 
        t^{\frac{1}{n}}
        |\nabla f|^{**}(t)
            =\sqrt{n} t^{\frac{1}{n}}
        \int_{t}^{\infty}
        \frac{|\nabla f|^{**}(s)-|\nabla f|^{*}(s)}{s}\dd s
            \\
        &\le 
        n t^{\frac{1}{n}}
        \int_{t}^{\infty}s^{\frac{1}{n}-1}|\nabla^{2} f|^{**}(s)\dd s,
    \end{align*}
hence the base case ($k=2$) of~\eqref{E:kolyada-BS} follows. Now, suppose that~\eqref{E:kolyada-BS} holds for some $k\in\N$, $k\ge 2$, and, once again, fix $t\in(0,\infty)$. Then, by~\eqref{E:formula-rearrangement}, Fubini's theorem and~\eqref{E:kolyada-BS}, one has
    \begin{align*}
        \int_{t}^{\infty}s^{\frac{k-1}{n}-1}|\nabla^{k} f|^{**}(s)\dd s
        &=
        \int_{t}^{\infty}
        s^{\frac{k-1}{n}-1}
        \int_{s}^{\infty}
        \frac{|\nabla^k f|^{**}(\tau)-|\nabla^k f|^{*}(\tau)}{\tau}\dd \tau\dd s
            \\
        &=
        \int_{t}^{\infty}
        \frac{|\nabla^k f|^{**}(\tau)-|\nabla^k f|^{*}(\tau)}
        {\tau}
        \int_{t}^{\tau}s^{\frac{k-1}{n}-1}\dd s
        \dd \tau
            \\
        &\le
        \int_{t}^{\infty}
        \frac{|\nabla^k f|^{**}(\tau)-|\nabla^k f|^{*}(\tau)}
        {\tau}
        \int_{0}^{\tau}s^{\frac{k-1}{n}-1}\dd s
        \dd \tau
            \\
        &=
        \frac{n}{k-1} 
        \int_{t}^{\infty}
        \big(|\nabla^k f|^{**}(\tau)-|\nabla^k f|^{*}(\tau)\big)
        \tau^{\frac{k-1}{n}-1}\dd \tau
            \\
        &\le 
        \frac{n^{\frac{3}{2}}}{k-1} 
        \int_{t}^{\infty}
        s^{\frac{k}{n}-1}|\nabla^{k+1} f|^{**}(s)\dd s.
    \end{align*}
    Plugging this back into~\eqref{E:kolyada-BS}, which still holds, with the fixed $k$, owing to the induction preposition, and using the fact that $t$ was arbitrarily chosen, we get 
    \begin{equation*}
        f^{**}(t)-f^*(t)
        \le
        \frac{n^{\frac32(k+1)-2}}{(k-1)!}
        t^{\frac{1}{n}}
        \int_{t}^{\infty}
        s^{\frac{k}{n}-1}
        |\nabla^{k+1} f|^{**}(s)\dd s
        \quad\text{for $t\in(0,\infty)$.}
    \end{equation*}
    This yields~\eqref{E:kolyada-BS} with $k$ replaced by $k+1$, and thence establishes the induction step. Altogether, the claim now follows, by the induction process, for general integer $k\ge 2$.
    
    Raising the terms on each side of~\eqref{E:kolyada-BS} to $p$, then multiplying each of them by $t^{-\frac{kp}{n}}$, integrating the resulting expressions with respect to $t$ over $(0,\infty)$, and finally taking the $p$-th roots, we arrive at
    \begin{align*}
        \left(\int_{0}^{\infty}\big(f^{**}(t)-f^*(t)\big)^{p}
        t^{-\frac{kp}{n}}\dd t\right)^{\frac{1}{p}}
        \le      
        \frac{n^{\frac32k-2}}{(k-2)!}
        \left(\int_{0}^{\infty}
        t^{\frac{(1-k)p}{n}}
        \left(\int_{t}^{\infty}
        s^{\frac{k-1}{n}-1}|\nabla^{k} f|^{**}(s)\dd s
        \right)^{p}
        \dd t\right)^{\frac{1}{p}}.
    \end{align*}
    We now apply the weighted Hardy inequality from~\cite[Chapter~V, Lemma 3.14(ii), p.~196]{Ste:71} with $q=p$ and $r=1+\frac{(1-k)p}{n}$ to the expression on the right-hand side of the last inequality. Note that in order to perform this step, we need the upper bound for $k$ from~\eqref{E:conditions-on-k} to hold. We obtain
    \begin{align*}
        \left(\int_{0}^{\infty}\big(f^{**}(t)-f^*(t)\big)^{p}
        t^{-\frac{kp}{n}}\dd t\right)^{\frac{1}{p}}
        \le      
        C_{n,k,p}
        \left(\int_{0}^{\infty}
        |\nabla^{k} f|^{**}(t)^{p}
        \dd t\right)^{\frac{1}{p}}
    \end{align*}
    with some positive constant $C_{n,k,p}$ depending only on the indicated parameters and, in particular, independent of $f$.
    This can be interpreted as the Sobolev embedding
    \begin{equation*}
        V^{k,p}(\rn)
        \to
        S^{p}(t^{-\frac{kp}{n}})(\rn),
    \end{equation*}   
    in which $V^{k,p}(\rn)$ denotes the \emph{homogeneous Sobolev space} defined as the collection of all weakly differentiable functions $u$ whose weak $k$-th gradient satisfies 
    $|\nabla^k u|\in L^p(\rn)$, endowed with the seminorm $\|\nabla^k u\|_{L^p(\rn)}$.
    We complement this embedding with the trivial one,
    \begin{equation*}
        L^{p}(\rn)
        \to
        L^{p}(\rn),
    \end{equation*}
    and we derive the corresponding $K$-inequality, obtaining thereby
    \begin{equation}
        \label{E:K-inequality}
        K(f,t^{\frac kn};L^{p}(\rn),S^{p}(t^{-\frac{kp}{n}})(\rn))
        \le C
        K(f,t^{\frac kn};L^{p}(\rn),V^{k,p}(\rn)).
    \end{equation}
    Now, by Corollary~\ref{C:1} with $\alpha=\frac{kp}{n}$, we get
    \begin{align}
        \label{E:term-on-left}
        &K(f,t^{\frac kn};L^{p}(\rn),S^{p}(t^{-\frac{kp}{n}})(\rn))
            \\
        &\qquad\approx
        \left(\int_{0}^{t}
        \left(f^{**}(s)-f^*(s)\right)^{p}\dd s
        \right)^{\frac{1}{p}}
        +
        t^{\frac{k}{n}}
        \left(\int_{t}^{\infty}
        \left(f^{**}(s)-f^*(s)\right)^{p}s^{-\frac{kp}{n}}\dd s
        \right)^{\frac{1}{p}},\nonumber
    \end{align}
    while from~\cite[Chapter~5, Estimate~(4.42), p.~341]{Ben:88} we infer that
     \begin{equation}
        \label{E:term-on-right}        
        K(f,t^{\frac kn};L^{p}(\rn),V^{k,p}(\rn))
        \ls
        \omega_k(f,t^{\frac{1}{n}})_p.
    \end{equation}
    The result now follows from~\eqref{E:K-inequality},~\eqref{E:term-on-left} and~\eqref{E:term-on-right}.
\end{proof}


\bibliographystyle{abbrv}
\bibliography{interpolation-of-S-spaces}

\end{document}